\documentclass[12pt]{article}

\usepackage{geometry} 
\usepackage{tikz}
\usepackage{adjustbox}

\usepackage{bbm}
\usepackage{amsmath, amsfonts, amssymb, amsthm, dsfont,mathrsfs,yfonts,yhmath,stmaryrd, upgreek}
\usepackage{verbatim}
\usepackage{microtype} 
\usepackage{graphicx}
\usepackage{afterpage}
\usepackage{gensymb}
\usepackage{color}
\usepackage{transparent}
\usepackage[french,english]{babel}
\usepackage[utf8]{inputenc}  
\usepackage[T1]{fontenc}  
\usepackage{pdfpages} 
\usepackage{mathtools}
\usepackage{comment} 

\usepackage[unicode,psdextra]{hyperref}
\usepackage{bookmark} 

\usepackage[labelsep=period]{caption} 

\newtheorem{lemma}{Lemma}
\newtheorem{theorem}[lemma]{Theorem}
\newtheorem{proposition}[lemma]{Proposition}
\newtheorem*{defin}{Definition}
\newtheorem{remark}[lemma]{Remark}
\newtheorem*{rem}{Remark}
\newtheorem*{notation}{Notation}

\newcommand{\eps}{\varepsilon}  
\newcommand{\Del}{\Delta}	
\newcommand{\alp}{\alpha}
\newcommand{\bet}{\beta}

\newcommand{\del}{\delta}
\newcommand{\om}{\omega}

\newcommand{\ro}{\varrho}

\newcommand{\dis}{\displaystyle}

\newcommand{\calA}{\mathcal{A}}
\newcommand{\sfA}{\mathsf{A}}

\newcommand{\sfB}{\mathsf{B}}
\newcommand{\calH}{\mathcal{H}}

\newcommand{\sfK}{\mathsf{K}}
\newcommand{\calL}{\mathcal{L}}

\newcommand{\sfM}{\mathsf{M}}

\newcommand{\frakq}{\mathfrak{q}}

\newcommand{\calT}{\mathcal{T}}

\newcommand{\sfU}{\mathsf{U}}

\newcommand{\calV}{\mathcal{V}}
\newcommand{\sfV}{\mathsf{V}}

\newcommand{\sfW}{\mathsf{W}}

\newcommand{\sfX}{\mathsf{X}}

\newcommand{\sfY}{\mathsf{Y}}
\newcommand{\bbC}{\mathbb{C}}
\newcommand{\bbN}{\mathbb{N}}
\newcommand{\bbR}{\mathbb{R}}
\newcommand{\bbRz}{\mathbb{R}_{\neq0}}
\newcommand{\bbRzu}{\mathbb{R}_{\neq0,1}}
\newcommand{\bbZ}{\mathbb{Z}}

\newcommand{\ot}{\otimes}

\newcommand{\ssi}{\Leftrightarrow}

\newcommand{\rar}{\Rightarrow}
\newcommand{\lac }{\left\{ }
\newcommand{\rac }{\right\} }

\newcommand{\pa}[1]{\left( #1 \right)}

\newcommand{\brkt }[1]{\left\{ #1 \right\} }

\newcommand{\udl}{\underline}
\numberwithin{equation}{section}

\DeclareMathOperator{\Hom}{Hom}	
\DeclareMathOperator{\End}{End}	
\DeclareMathOperator{\Ker}{Ker}

\DeclareMathOperator{\Id}{Id}
\DeclareMathOperator{\PGL}{PGL}

\begin{document}
\date{}
    \title{$\mathfrak{sl}_3$ Matrix Dilogarithm as a $6j$-Symbol}
     \author{Mucyo Karemera
     \\
      \small{ Department of Mathematics and Statistics,
        Auburn University}
    }
\maketitle
\begin{abstract}
\noindent
We construct quantum invariants of 3-manifolds based on a $\mathfrak{sl}_3$ matrix dilogarithm proposed by Kashaev. This matrix dilogarithm is an $\mathfrak{sl}_3$ analogue of the (cyclic) quantum dilogarithm used to define Kashaev's invariants as well as Baseilhac and Benedetti's quantum hyperbolic invariants. 
In this article, we show that the $\mathfrak{sl}_3$ matrix dilogarithm can be considered as a 6$j$-symbol associated to modules of a quantum group related to $U_q(\mathfrak{sl}_3)$. Moreover, we show that the quantum invariants aforementioned allow to define a $\mathfrak{sl}_3$ version of Kashaev's invariants, opening a route to define a $\mathfrak{sl}_3$ version of Baseilhac and Benedetti's quantum hyperbolic invariants.

\end{abstract}%


\selectlanguage{english} 





\section{Introduction}

Turaev and Viro first observed in \cite{turaev1992state} that the category of representations of the quantum group $U_q(\mathfrak{sl}_2)$ gives rise to topological invariants of 3-manifolds. The invariants are obtained as state sums on triangulated 3-manifolds. The key ingredients of the state sums are the $6j$-symbols.  The $6j$-symbols are naturally associated with combinatorial tetrahedra. In the case of finite dimensional representation theory of the angular momentum, as well as of its $q$-deformation, the numerical values of $6j$-symbols are specified by six irreducible representations associated with the edges of the tetrahedron. In particular, they realize explicitly the tetrahedral symmetries. In the Turaev-Viro theory, a specification of the deformation parameter to roots of unity allows to separate a sector in the representation category with a finite set of irreducible representations, and the 3-manifold invariant is obtained by summing over all labelings of the edges of a triangulation by representations from this finite set.

A related construction to the Turaev-Viro invariants was made by Kashaev in \cite{kashaev1994quantum} and \cite{kashaev1995link}. Kashaev's invariants are defined as state sums on triangulations of the pair $(M,L)$, where $M$ is a 3-manifold and $L\subset M$ is a link, using charged versions of $6j$-symbols associated to certain modules of a Borel subalgebra of $U_q(\mathfrak{sl}_2)$. These invariants led him to the volume conjecture \cite{kashaev1997hyperbolic}. The particular feature of his $6j$-symbols is that they are constructed using the quantum dilogarithm function defined by Faddeev and Kashaev in \cite{faddeev1994quantum}. Moreover, these $6j$-symbols depend on a continuous parameter and the tetrahedral symmetries are realized implicitly through non-trivial transformation matrices. %
In parallel, Kashaev proposed in \cite{kashaev1999matrix}  matrix generalizations of the Rogers dilogarithm associated to associative algebras. He showed that the quantum dilogarithm is an example of such matrix dilogarithm associated to a Borel subalgebra of $U_q(\mathfrak{sl}_2)$ and he exhibited a matrix dilogarithm that is associated to a quantum group related to $U_q(\mathfrak{sl}_3)$. We refer to the latter as the $\mathfrak{sl}_3$ matrix dilogarithm.

Kashaev's invariants have been generalized in two ways. A topological generalization has been accomplished by Baseilhac and Benedetti in \cite{baseilhac2004quantum} where they define quantum hyperbolic invariants of 3-manifolds also based on the quantum dilogarithm. They interpret the complexified continuous parameters entering the $6j$-symbols as shape parameters characterizing isometry classes of ideal tetrahedra. 
This enables them to construct quantum invariants which are functions on deformation varieties of cusped 3-manifolds. 

This topological framework has been used, to some extent, in the other generalization of Geer, Kashaev and Turaev in \cite{geer2012tetrahedral}. This construction is a categorical generalization of the Kashaev-Baseilhac-Benedetti construction. It introduces the notion of a $\hat{\Psi}$-system in a monoidal abelian category which provides a general framework for charged $6j$-symbols. However, in this general context, the dependence of $6j$-symbols on continuous variables does not necessarily reduce only to one variable. Thus the interpretation in terms of the shape of hyperbolic ideal tetrahedra is not evident.

In this article, we use the framework proposed by Geer, Kashaev and Turaev in \cite{geer2012tetrahedral}, to construct quantum invariants of 3-manifolds based on the $\mathfrak{sl}_3$ matrix dilogarithm defined by Kashaev in \cite{kashaev1999matrix}. 
We also show that the $6j$-symbols involved in this construction, defined using this matrix dilogarithm, are similar to the one used in the Kashaev-Baseilhac-Benedetti theory in the sense that they also only depend on one parameter, allowing interpretation in terms of shape parameters of ideal hyperbolic tetrahedra. These results allow us to define a $\mathfrak{sl}_3$ analogue of Kashaev's invariants and
leads us to expect that a $\mathfrak{sl}_3$ version of Baseilhac and Benedetti's quantum hyperbolic invariants can be constructed. 


The paper is organized as follows. In Section \ref{hopf.algebra.A_w,t.red.cycl.mod}, we define the Hopf algebra $\calA_{\om,t}$ and consider a family of $\calA_{\om,t}$-modules. In Section \ref{sec.psi.sys.in.Awt}, we show that this family allows us to define a $\Psi$-system in the category of $\calA_{\om,t}$-modules. The $\mathfrak{sl}_3$ matrix dilogarithm play a key role in this construction. In Sections \ref{sec.op.in.H} we define operators that will allow us to extend the $\Psi$-system into a $\hat\Psi$-system and in Section \ref{sec.6j.symb}, we defined the $6j$-symbols using
(again) the $\mathfrak{sl}_3$ matrix dilogarithm. Finally, in Section \ref{sec.hat.psi.sys.in.Awt}, we prove the existence of the quantum invariants and we define the associated state sum in Section \ref{hyperbolic.invariant}. We also show, in Section \ref{hyperbolic.invariant}, how these invariants can be viewed as an $\mathfrak{sl}_3$ analogue of Kashaev's invariants. 






\section{The Hopf algebra \texorpdfstring{$\calA_{\om,t}$}{A\_\{\unichar{"03C9},t\}} and its reduced cyclic modules}\label{hopf.algebra.A_w,t.red.cycl.mod}
The two-paramerter quantum groups $U_{r,s}(\mathfrak{sl}_n)$ have been introduced by Takeuchi in \cite{takeuchi1990two}. In \cite{benkart2004two}, $r$ and $s$ are non zero elements in a filed $\mathbb{K}$ such that $r\neq s$ and $U_{r,s}(\mathfrak{sl}_n)$ is defined as the unital associative algebra over $\mathbb{K}$ generated by elements $k_i^{\pm 1}, (k'_i)^{\pm 1}, e_i, f_i $ $(1\le i< n)$ which satisfy the  following relations  

\begin{itemize}

\item[(R1)] The $k_i^{\pm 1}, (k'_j)^{\pm 1}$ all commutes with one another and $k_ik_i^{-1}=k'_j(k'_j)^{-1}=1$, \\

\item[(R2)] $k_ie_j=r^{\del_{i,j}-\del_{i,j+1}}s^{\del_{i+1,j}-\del_{i,j}}e_jk_i$ \ \ \text{and} \ \ $k_if_j=r^{\del_{i,j+1}-\del_{i,j}}s^{\del_{i,j}-\del_{i+1,j}}f_jk_i$, \\

\item[(R3)] $k'_ie_j=r^{\del_{i+1,j}-\del_{i,j}}s^{\del_{i,j}-\del_{i,j+1}}e_jk'_i$ \ \ \text{and} \ \ $k'_if_j=r^{\del_{i,j}-\del_{i+1,j}}s^{\del_{i,j+1}-\del_{i,j}}f_jk'_i$, \\ 

\item[(R4)] $[e_i,f_j]=\dfrac{\del_{i,j}}{r-s}(k_i-k'_i)$, \\

\item[(R5)] $[e_i,e_j]=[f_i,f_j]=0$ if $|i-j|>1$, \\

\item[(R6)] $e_i^2e_{i+1} = (r+s)e_ie_{i+1}e_i  +rse_{i+1}e_i^2$,
			\\
		 $e_ie_{i+1}^2 = (r+s)e_{i+1}e_ie_{i+1} + rse_{i+1}^2e_i$, 	\\

\item[(R7)]  $f_i^2f_{i+1} = (r^{-1}+s^{-1})f_if_{i+1}f_i  +(rs)^{-1}f_{i+1}f_i^2$,
			\\
		 $f_if_{i+1}^2 = (r^{-1}+s^{-1})f_{i+1}f_if_{i+1} + (rs)^{-1}f_{i+1}^2f_i$,

\end{itemize}
where $\delta$ is Kronecker's delta.
The algebra $U_{r,s}(\mathfrak{sl}_n)$ is a Hopf algebra with coproduct $\Del$ defined by 
\begin{align*}
\Del\left(k_i\right) &= k_i \ot k_i   &\Del(e_i)  &= e_i\ot 1+k_i\ot e_i, 
\\
\Del\left(k'_i\right) &= k'_i \ot k'_i,  &\Del(f_i)  &= 1\ot f_i+f_i\ot k'_i,
\end{align*}
and counit $\eps$ and antipode $\gamma$ defined by
\begin{align*}
     \eps(k_i)=\eps(k'_i) &= 1, & \gamma(k_i) &= k_i^{-1}, & \gamma(k'_i) &=  (k'_i)^{-1},
\\
  \eps(e_i)=\eps(f_i) &= 0, & \gamma(e_i) &= -k_i^{-1}e_i, & \gamma(f_i) &= -f_i(k'_i)^{-1}. 
\end{align*}

When $r=q$ and $s=q^{-1}$, $U_{r,s}(\mathfrak{sl}_n)$ modulo the Hopf ideal generated by the elements $k'_i-k_i^{-1}, 1\le i <n$, is $U_{q}(\mathfrak{sl}_n)$.

\smallskip
The algebra $\calA_{\om,t}$ that we are about to introduce, is obtained as a certain quotient of the Borel subalgebra $BU_{r,s}(\mathfrak{sl}_3)$ of $U_{r,s}(\mathfrak{sl}_3)$ generated by $k_i^{\pm 1}, e_i$, for a certain choice of $r,s\in\bbC$.

\subsection{The Hopf algebra \texorpdfstring{$\calA_{\om,t}$}{A\_\{\unichar{"03C9},t\}}}

We fix an integer $t$, a natural number $N\notin3\bbN$ which divides $t^2+t+1$ and a primitive $N$-th root of unity $\om$. Remark that $N$ and $t$ are relatively prime and $N$ is odd. From the latter, we can define $\om^{\frac{1}{2}}=\om^{\frac{N+1}{2}}$. In the sequel, for any $a\in\frac{1}{2}\bbZ$, we will simply write $\om^a$ instead of $\om^{a(N+1)}$ and we will also write $\bbZ_N$ instead of $\bbZ/N\bbZ$.

\smallskip
The Hopf algebra $\calA_{\om,t}$ is defined as the quotient of $BU_{\om^t,\om^{t+1}}(\mathfrak{sl}_3)$ by the Hopf ideal generated by $k_1^N-1$ and $k_2-k_1^t$. Therefore, the generators $k_1,e_1$ and $e_2$ satisfy the following relations
\begin{equation}\begin{array}{c}\label{rel.Awt}
  k_1^N = 1, \ \ k_1e_1 = \om^{-1}e_1k_1, \ \ k_1e_2 =  \om^{t+1}e_2k_1,
\\
  e_1^2e_2 = (\om^t+\om^{t+1})e_1e_2e_1  +\om^{2t+1}e_2e_1^2 , 
\\
  e_1e_2^2 = (\om^t+\om^{t+1})e_2e_1e_2 + \om^{2t+1}e_2^2e_1.  
\end{array}\end{equation}
Note that given our assumptions on $N$ and $t$, one can deduce that all the powers of $\om$ in \eqref{rel.Awt} have a multiplicative inverse in $\bbZ_N$.
The coproduct $\Del : \calA_{\om,t} \to \calA_{\om,t}\ot \calA_{\om,t}$ is given by
\begin{align}\label{Del.Awt}
 \Del(k_1)  = k_1\ot k_1,  \ \  \Del(e_1)  = e_1\ot 1+k_1\ot e_1, \ \ \Del(e_2)  = e_2\ot 1+k_1^t\ot e_2,
\end{align}
the counit $\eps : \calA_{\om,t} \to \bbC$ by
\begin{align}\label{counit.Awt}
 \eps(k_1)  = 1, \ \ \eps(e_1) = 0, \ \ \eps(e_2)  = 0, 
\end{align}
and the antipode $\gamma : \calA_{\om,t} \to \calA_{\om,t}$ by
\begin{align}\label{antip.Awt}
\gamma(k_1) = k_1^{-1}, \ \ \gamma(e_1)  = -k_1^{-1}e_1, \ \ \gamma(e_2)  = -k_1^{-t}e_2. 
\end{align}

\subsection{Reduced cyclic \texorpdfstring{$\calA_{\om,t}$}{A\_\{\unichar{"03C9},t\}}-modules}
In what follows, all the vector spaces are finite dimensional $\bbC$-vector spaces. We will also use the following notations:
%

\begin{enumerate}
   \item $\calV = \bbC^N\ot\bbC^N$ and\, $\calA = \End(\calV)$,
    \item $\left\{v_\alp\right\}_{\alp\in\bbZ_N}$ is the canonical basis of $\bbC^N$, where the index $\alpha$ taking its values in the set $\{0,1,\dots,N-1\}$ is interpreted as an element of $\bbZ_N$,
   
    \item the canonical basis $\{v_\alpha\ot v_\beta\}_{\alpha,\beta\in\bbZ_N}$ of $\calV$ is denoted $\{v_{(\alpha,\beta)}\}_{(\alpha,\beta)\in\bbZ_N^2}$, and we do similarly with any other basis $\left\{u_\alp\ot u_\beta\right\}_{\alp,\beta\in\bbZ_N}$ of $\calV$.
    
\end{enumerate}

Let $\sfX,\sfY\in\End(\bbC^N)$ be the invertible operators defined by
\begin{align}\label{commXY}
\sfX v_\alp=\om^{-\alp}v_\alp, \ \ \sfY v_\alp=v_{\alp+1}.
\end{align}
\begin{lemma}\label{algebra.A}
The algebra $\End(\bbC^N)$ is generated by $\brkt{\sfX,\sfY}$. In particular, $\calA$ is generated by $\sfX_1,\sfX_2,\sfY_1,\sfY_2$ where
\begin{align*}
\sfX_1=\sfX\ot\Id_{\bbC^N}, \ \ \sfY_1=\sfY\ot\Id_{\bbC^N}, \ \ \sfX_2=\Id_{\bbC^N}\ot \sfX, \ \ \sfY_2=\Id_{\bbC^N}\ot \sfY.
\end{align*}  
\end{lemma}
\begin{proof}
We consider the canonical basis $\left\{v_\alp\right\}_{\alp\in\bbZ_N}$ of $\bbC^N$ and its dual basis $\left\{\bar{v}_\alp\right\}_{\alp\in\bbZ_N}$. The set $\brkt{E_{i,j}}_{i,j\in\bbZ_N}\subset \End(\bbC^N)$ defined by  
\begin{align*}
\left< \bar{v}_\bet | E_{i,j} | v_\alp\right>= \del_{\alp,j}\del_{\bet,i}, \ \ \forall i,j,\alp,\bet\in\bbZ_N,
\end{align*}
is a linear basis of $\End(\bbC^N)$. It is easy to check that we have the following relation
\begin{align}\label{rel.Eij}
E_{i,j}E_{k,l}=\del_{j,k}E_{i,l}, \ \ \forall i,j,k,l\in\bbZ_N.
\end{align}

First, we compute, for $i,\alp,\bet\in\bbZ_N$,
\begin{align*}
\frac{1}{N}\sum_{j\in\bbZ_N}\left<\bar{v}_\bet | \om^{ij}\sfX^j | v_\alp\right> &=  \frac{1}{N}\sum_{j\in\bbZ_N} \om^{ij}\om^{-\alp j}\del_{\alp,\bet} 
=  \frac{1}{N}\sum_{j\in\bbZ_N} \om^{j(i-\alp)}\del_{\alp,\bet}  
\\
&=  \frac{1}{N}N\del_{i,\alp}\del_{\alp,\bet} = \left<\bar{v}_\bet | E_{i,i} | v_\alp\right>.
\end{align*}
Hence, for all $i\in\bbZ_N$ we have $\displaystyle E_{i,i}=\frac{1}{N}\sum_{j\in\bbZ_N}\om^{ij}\sfX^j$.

Now, using relation \eqref{rel.Eij} and since $\sfY=\sum_{j\in\bbZ_N}E_{j+1,j}$, we can compute, for all $j,k\in\bbZ_{N}$
\begin{align*}
E_{k,k}\sfY E_{i,i} =  E_{k,k}\sum_{j\in\bbZ_N}E_{j+1,j}E_{i,i} = E_{k,k}\sum_{i\in\bbZ_N}\del_{i,j}E_{j+1,i}= E_{k,k}E_{i+1,i} 
=\del_{k,i+1}E_{i+1,i}.
\end{align*} 
Hence, for all $i\in\bbZ_N$, $E_{i+1,i}=E_{i+1,i+1}\sfY E_{i,i}$. 

Finally, for any $i,j\in\bbZ_N$ we have 
\begin{align*}
E_{i,i-1}E_{i-1,i-2}\cdots E_{i-(i-j)+1,i-(i-j)}=E_{i,i-1} E_{i-1,i-2}\cdots E_{j+1,j}=E_{i,j}.
\end{align*}
Therefore, the operators $\sfX$ and $\sfY$ generate $\End(\bbC^N)$. 
\end{proof}
We now consider the following operators in $\mathcal{A}$
\begin{align}\label{def.el.calA}
X=\sfX_1, \ \ Y=\sfY_1, \ \ U=\sfY_1^{-t}\sfX_1^{t}\sfX_2^{-t}, \ \ V=\sfY_1^{-t}\sfX_1^{t+1}\sfY_2.
\end{align}
We can easily see that we have $X^N=Y^N=U^N=V^N=\Id_\calV$. Moreover, by Lemma \ref{algebra.A}, $\mathcal{A}$ is generated by $\brkt{X,Y,U,V}$ since 
\begin{align*}
\mathsf{X}_1=X, \ \ \mathsf{Y}_1=Y, \ \ \sfX_2=(X^{-t}Y^tU)^{t+1}, \ \ \sfY_2=X^{-t-1}Y^tV.
\end{align*}

The operators in \eqref{def.el.calA} are used in the following Proposition to define a $\calA_{\om,t}$-module structure on $\calV$. In the sequel, $\bbR\backslash\{0\}$ will be denoted by $\bbRz$.

\begin{proposition}\label{prop.Awt.module.structure}
For any $p\in\bbR_{\neq0}$, let $V_p$ be the space $\calV$ provided with a $\calA_{\om,t}$-module structure $\pi_p:\calA_{\om,t}\to\End(V_p)$ defined by
 \begin{equation}\begin{split}\label{act.Awt}
	\begin{array}{l}
  	\pi_p(k_1)  =  X, \ \ \pi_p(e_1)  =  p^\frac{1}{N}Y,
	\\
	 \pi_p(e_2)  =  \left(\frac{1}{2}p\right)^\frac{1}{N}\pa{U+V}Y^{-1}, 
  	 \end{array}
  \end{split} \end{equation}
where the $N$th root is chosen to be the unique real root. Furthermore, $V_p$ is cyclic, i.e. the operators $\pi_p(k_1),\pi_p(e_1)$ and $\pi_p(e_2)$ are invertible. 
\end{proposition}
\begin{defin}
We call $V_p$ a {\it reduced cyclic $\calA_{\om,t}$-module of parameter} $p$. 
\end{defin}
\begin{rem}
The reason why we have chosen to use the word ``reduced'' to name the cyclic $\calA_{\om,t}$-modules $V_p$ is because the set $\brkt{V_p}_{p\in\bbRz}$ is included in a much larger set of cyclic $\calA_{\om,t}$-modules. Indeed, for any $\mathfrak{p}=(p_1,p_2,p_3)\in\bbC^3$ such that $p_1\in\bbC^*$ and $p_2\notin\left\{-\om^ip_3|i\in\bbZ_N\right\}$, the space $\calV$ can be provided with a cyclic $\calA_{\om,t}$-module structure $\pi_{\mathfrak{p}}:\calA_{\om,t}\to\End(\calV)$, defined by
\begin{equation*}\begin{split}
  	\begin{array}{l} 
      	\pi_\mathfrak{p}(k_1)  =  X, \ \ \pi_\mathfrak{p}(e_1)  =  p_1Y
	\\
      	\pi_\mathfrak{p}(e_2)  =  p_1^{-t-1}(p_2U+p_3V)Y^{-1}.
  	 \end{array} 
  \end{split} \end{equation*}
The set of reduced cyclic $\calA_{\om,t}$-module is determined by triples of the form $\mathfrak{p}=\pa{p^\frac{1}{N},\pa{\frac{1}{2}p^2}^\frac{1}{N},\pa{\frac{1}{2}p^2}^\frac{1}{N}}\in\bbR^3$ where $p\in\bbRz$.
\end{rem}
Before the proof, we recall that for any variables $x$ and $y$ satisfying the relation $yx=\om xy$, the {\it $q$-binomial formula} (see e.g. \cite[Proposition IV.2.2]{kassel1995quantum}) implies that  
\begin{align*}
			(x+y)^N=x^N+y^N,
\end{align*} 
since $\om$ is a $N$-th root of unity. This formula will be used in the sequel.

\begin{proof}[Proof of Proposition \ref{prop.Awt.module.structure}]
Using equalities \eqref{commXY}, one can check that 
\begin{align*}
\sfX\sfY=\om^{-1}\sfY\sfX.
\end{align*} 
Then, using equalities \eqref{def.el.calA}, one easily sees that $\pi_p$ is an algebra morphism.

The operators $\pi_p(k_1)$ and $\pi_p(e_1)$ are clearly invertible since $p\in\bbRz$ and $\sfX,\sfY\in\End(\bbC^N)$ are invertible operators. For the invertibility of $\pi_p(e_2)$ we compute $\pi_q(e_2)^N$. Since 
\begin{align*}
UY^{-1}\cdot VY^{-1}=\om^{-1}VY^{-1}\cdot UY^{-1}
\end{align*}
we can compute $\pi_q(e_2)^N$ using the $q$-binomial formula. Indeed, we have
\begin{align}\label{e2N}\begin{split}
\pi_p(e_2)^N  &= \tfrac{1}{2}p\left[UY^{-1}+VY^{-1}\right]^N 
= \tfrac{1}{2}p\left[\left(UY^{-1}\right)^N+\left(VY^{-1}\right)^N\right] 
\\
&= \tfrac{1}{2}p\left[\Id_{\calV}+\Id_{\calV}\right] = p\Id_{\calV}.
\end{split}
\end{align}
We conclude that the operator $\pi_p(e_2)$ is invertible since $p\in\bbRz$. 
\end{proof}

In order to study the reduced cyclic $\calA_{\om,t}$-modules, we need to introduce particular elements of $\calA_{\om,t}$. First, we define $a_1,a_2\in \calA_{\om,t}$ by
\begin{align}\label{def.a}\begin{split}
 \om^t(1-\om)a_1 & = e_1e_2-\om^{t+1}e_2e_1,
 \\
 (1-\om)a_2 & = e_2e_1-\om^{-t}e_1e_2.
\end{split}\end{align}
Using \eqref{rel.Awt} and \eqref{def.a} one can check that 
\begin{equation}\label{rel.uv}\begin{array}{ll}
k_1a_1 = \om^ta_1k_1,\qquad   & k_1a_2 = \om^ta_2k_1, \\
 e_1a_1 = \om^ta_1e_1,  & e_1a_2 = \om^{t+1}a_2e_1, \\
 e_2a_1 = \om^{-t}a_1e_2,  & e_2a_2 = \om^{-(t+1)}a_2e_2, \\
 a_1a_2 = a_2a_1, & e_2e_1 = a_1+a_2.
\end{array}\end{equation}

and with \eqref{Del.Awt} we have
\begin{align}\label{Del.a}\begin{split}
\Del(a_1) = a_1\ot 1 + k_1^{t+1}\ot a_1 + k_1^te_1\ot e_2, 
\\
\Del(a_2) = a_2\ot 1+ k_1^{t+1}\ot a_2 + e_2k_1\ot e_1.
\end{split}\end{align}
Next, we consider elements $c_1,c_2\in \calA_{\om,t}$ defined by 
\begin{align}\label{def.xy}
c_1 = a_1k_1^{-t}e_1^t, \ \ c_2= a_2k_1^{-(t+1)}e_1^t.
\end{align}
By using \eqref{rel.Awt} and \eqref{rel.uv}, we deduce the following relations
\begin{equation*}\begin{array}{llc}
k_1c_1 = c_1k_1, & k_1c_2 = c_2k_1,
\\
e_1c_1 = c_1e_1, & e_1c_2 = c_2e_1, & c_1c_2=\om^tc_2c_1
\end{array}\end{equation*}

Finally, for a reduced cyclic $\calA_{\om,t}$-module $V_p$, using \eqref{act.Awt}, \eqref{def.a} and \eqref{def.xy}, we have 
\begin{equation}\label{pi.a}\begin{array}{ll}
 \pi_p(a_1)  =  \left(\frac{1}{2}p^2\right)^\frac{1}{N}U,  &  \pi_p(a_2)  =  \left(\frac{1}{2}p^2\right)^\frac{1}{N}V, \\
 \\
 \pi_p(c_1) = \left(\frac{1}{2}p^{t+2}\right)^\frac{1}{N}\sfX_2^{-t},  & \pi_p(c_2) = \left(\frac{1}{2}p^{t+2}\right)^\frac{1}{N}\sfY_2.
\end{array}\end{equation}
Note that $\lac v_\alp\rac_{\alp\in\bbZ_N^2}$ is a basis of common eigenvectors of $\pi_p(k_1)$ and $\pi_p(c_1)$.

\begin{lemma}\label{equiv.rep} A reduced cyclic $\calA_{\om,t}$-module $V_p$ is simple and $V_q$ is equivalent to $V_p$ only if $p=q$. 
\end{lemma}
\begin{proof}
The simplicity is clear. Indeed, since 
\begin{equation*}\begin{array}{ll}
  \pi_p(k_1) = \sfX_1,    \ \ &\pi_p(c_1^{t+1}) = \left(\frac{1}{2}p^{t+2}\right)^\frac{t+1}{N}\sfX_2, \\
 \\
\pi_p(e_1)=p^\frac{1}{N}\sfY_1,  \ \ &\pi_p(c_2) = \left(\frac{1}{2}p^{t+2}\right)^\frac{1}{N}\sfY_2,
\end{array}\end{equation*}
by Lemma \ref{algebra.A}, the only invariant subspaces of $V_p$ are $\lac 0\rac$ and $V_p$. 

If $V_p\cong V_q$ then, by definition, there is an isomorphism $S: V_p\to V_q$ such that for all $a\in \calA_{\om,t}$, $S\pi_p(a)=\pi_q(a)S$. In particular, we have 
\begin{align*}
S(p\Id_\calV)=S\pi_p(e_1^N)=\pi_q(e_1^N)S=(q\Id_\calV)S\rar p=q.
\end{align*}
\end{proof}

Let $\Hom_{\calA_{\om,t}}(V_p,V_q)$ be the set of morphisms of $\calA_{\om,t}$-modules between $V_p$ and $V_q$ i.e., the set of linear maps $f:V_p\to V_q$ such that for all $a\in A_{\om,t}$, we have
\begin{align*} 
 \pi_q(a) f=f \pi_p(a).
\end{align*}
Lemma \ref{equiv.rep} and Schur's Lemma imply that 
\begin{enumerate}

	\item if $p\neq q$ then $\Hom_{\calA_{\om,t}}(V_p,V_q)=0$,
	
	\item if $p=q$ then $\Hom_{\calA_{\om,t}}(V_p,V_p)=\End_{\calA_{\om,t}}(V_p)=\bbC\Id_\calV$.
	
\end{enumerate}
Hence, for any $f\in\End_{\calA_{\om,t}}(V_p)$, there is a unique $c\in\bbC$ such that $f=c\Id_\calV$. This number is denoted $\left<f\right>$ in the sequel.
%

\subsection{Tensor product of reduced cyclic \texorpdfstring{$\calA_{\om,t}$}{A\_\{\unichar{"03C9},t\}}-modules}
The tensor product $V_p\ot V_q$ is provided with a $\calA_{\om,t}$-module structure through 
\begin{align*}
(\pi_p\ot \pi_q)\Del : \calA_{\om,t}\to\End(V_p\ot V_q).
\end{align*} 

\begin{defin}
An {\it admissible pair} $(V_p,V_q)$ is a pair of reduced cyclic $\calA_{\om,t}$-modules such that $p\not =-q$. In such case, we will also say that the pair $(p,q)$ is admissible. 
\end{defin}

The reason we are interested in admissible pairs is that their tensor product decomposes as a direct sum of reduced cyclic $\calA_{\om,t}$-modules. 

\begin{lemma}\label{lemma.tens}
We have 
 \begin{align*}
(\pi_p\ot \pi_q)\Del(e_1)^N & =(p+q)\Id_\calV\ot\Id_\calV, \\
(\pi_p\ot \pi_q)\Del(c_1)^N & =(\pi_p\ot \pi_q)\Del(c_2)^N=\tfrac{1}{2}(p+q)^{t+2}\Id_\calV\ot\Id_\calV.
\end{align*} 
\end{lemma}
\begin{proof}
Let $(V_p,V_q)$ be a pair of reduced cyclic $\calA_{\om,t}$-modules. 
For $(\pi_p\ot \pi_q)\Del(e_1)$ we compute
\begin{align*}
(\pi_p\ot\pi_q)\Del(e_1)^N  &= (\pi_p\ot\pi_q)\pa{\Del(e_1)^N} = (\pi_p\ot\pi_q)\pa{(e_1\ot 1+ k_1\ot e_1)^N} 
\\
&= (\pi_p\ot\pi_q)\pa{e_1^N\ot 1+ k_1^N\ot e_1^N}= pY^{N}\ot\Id_\calV+qX^{N}\ot Y^N 
\\
&=(p+q)\Id_\calV\ot\Id_\calV
\end{align*}
where we used the $q$-binomial formula for the third equality. 

For  $(\pi_p\ot \pi_q)\Del(c_1)$ we have
\begin{align*}
(\pi_p\ot\pi_q)\Del(c_1)^N  &= (\pi_p\ot\pi_q)\Del\left((a_1k_1^{-t}e_1^{t})^N\right) = (\pi_p\ot\pi_q)\Del\left(a_1^Nk_1^{-Nt}e_1^{Nt}\right) 
\\
& = (\pi_p\ot\pi_q)\Del(a_1^N)(\pi_p\ot\pi_q)\Del(k_1^N)^{-t}(\pi_p\ot\pi_q)\Del(e_1^N)^t  
\\ 
& = (\pi_p\ot\pi_q)\Del(a_1)^N(\Id_\calV\ot\Id_\calV)^{-t}(p+q)^t\Id_\calV\ot\Id_\calV 
\\
& = (\pi_p\ot\pi_q)\Del(a_1)^N(p+q)^t\Id_\calV\ot\Id_\calV
\end{align*}
where we used \eqref{rel.Awt} and \eqref{rel.uv} for the second equality. Using \eqref{rel.uv} again, one can use the $q$-binomial formula to compute $(\pi_p\ot\pi_q)\Del(a_1)^N$ since $t$ and $t+1$ are invertible in $\bbZ_N$. Using \eqref{Del.a}, we have
\begin{align*}
(\pi_p\ot \pi_q)\Del(a_1)^N  &= (\pi_p\ot \pi_q)\pa{ a_1\ot 1 + k_1^{t+1}\ot a_1 + k_1^te_1\ot e_2 }^N 
\\
& = (\pi_p\ot \pi_q)\pa{ a_1^N\ot 1 + k_1^{N(t+1)}\ot a_1^N + k_1^{Nt}e_1^N\ot e_2^N }
 \\
& =   \tfrac{1}{2}p^2U^N\ot\Id_\calV+\tfrac{1}{2}q^2\Id_\calV\ot U^N+pq\Id_\calV\ot\Id_\calV 
 \\ 
&  = \tfrac{1}{2}(p+q)^2\Id_\calV\ot\Id_\calV			      
\end{align*}
where we used \eqref{rel.Awt}, \eqref{rel.uv} and the $q$-binomial formula for the second equality, \eqref{e2N} and \eqref{pi.a} for the third one and \eqref{rel.Awt} and \eqref{act.Awt} for the last one. Hence we have 
\begin{align*}
(\pi_p\ot\pi_q)\Del(c_1)^N  = \tfrac{1}{2}(p+q)^{t+2}\Id_\calV\ot\Id_\calV.
\end{align*}

Finally the computation of $(\pi_p\ot\pi_q)\Del(c_2)^N$ is similar to the previous one. We have 
\begin{align*}
(\pi_p\ot\pi_q)\Del(c_2)^N  &= (\pi_p\ot\pi_q)\Del\left( (a_2k_1^{-(t+1)}e_1^{t})^N\right) = (\pi_p\ot\pi_q)\Del\left(a_2^Nk_1^{-N(t+1)}e_1^{Nt}\right) 
\\
& = (\pi_p\ot\pi_q)\Del(a_2)^N(p+q)^t\Id_\calV\ot\Id_\calV 
 \\
& =   (\pi_p\ot\pi_q)(a_2\ot 1+ k_1^{t+1}\ot a_2 + e_2k_1\ot e_1)^N(p+q)^t\Id_\calV\ot\Id_\calV 
 \\
& = (\pi_p\ot\pi_q)(a_2^N\ot 1+ k_1^{N(t+1)}\ot a_2^N + e_2^Nk_1^N\ot e_1^N)(p+q)^t\Id_\calV\ot\Id_\calV
 \\	
& =  \tfrac{1}{2}(p+q)^{(t+2)}\Id_\calV\ot\Id_\calV.			
\end{align*}

\end{proof}

\begin{proposition}\label{prod.tens}
Let $(V_p,V_q)$ be an admissible pair. Then the $\calA_{\om,t}$-module $V_p\ot V_q$ is equivalent to the direct sum of $N^2$ copies of the reduced cyclic $\calA_{\om,t}$-module $V_{p+q}$.
\end{proposition}

\begin{proof}
In order to find the submodules of $V_p\ot V_q$, we only consider the action of $k_1,c_1,e_1$ and $c_2$ on $V_p\ot V_q$.

First, $(\pi_p\ot \pi_q)\Del(k_1)=X\ot X$ is clearly diagonalizable and its spectrum is the set of all $N$-th roots of unity $\left\{\om^{\alp} | \alp\in\bbZ_N\right\}$. We write 
\begin{align*}
V_p\ot V_q=\bigoplus_{\alp\in\bbZ_N} W_{\alp}
\end{align*}
where $W_{\alp}=\Ker\pa{(\pi_p\ot\pi_q)\Del(k_1)-\om^{\alp}\Id_\calV\ot\Id_\calV}$. We have $\dim W_{\alp}=N^3$ for all $\alp\in\bbZ_N$. 

Now we show that for each $\alp\in\bbZ_N$, we can decompose $W_{\alp}$ in the following way
\begin{align*}
W_{\alp}=\bigoplus_{\bet\in\bbZ_N} W_{(\alp,\bet)}
\end{align*}
where $W_{(\alp,\bet)}\! =\! \Ker\pa{(\pi_p\ot\pi_q)\Del(c_1)|_{_{W_{\alp}}}-s\om^{t\bet}\Id_{W_{\alp}}}$ and $s\!=\!\pa{\tfrac{1}{2}(p+q)^{t+2}}^\frac{1}{N}\! \in\!\bbRz$.
By Lemma \ref{lemma.tens}, we have 
\begin{align*}
\pa{(\pi_p\ot\pi_q)\Del(c_1)}^N = (\pi_p\ot\pi_q)\Del(c_1)^N  = s^N\Id_\calV\ot\Id_\calV.
\end{align*}
This means that the minimal polynomial of $(\pi_p\ot\pi_q)\Del(c_1)$ divides %
\begin{align*}
x^N-s^N=\prod_{\bet\in\bbZ_N}(x-s\om^{\bet}).
\end{align*}
This implies that the minimal polynomial of $(\pi_p\ot\pi_q)\Del(c_1)$ has only simple zeros, which means that $(\pi_p\ot\pi_q)\Del(c_1)$ is diagonalizable. Moreover, it also implies that the spectrum of $(\pi_p\ot\pi_q)\Del(c_1)$ is contained in $\left\{s\om^{\bet} | \bet\in\bbZ_N\right\}$. Actually, its spectrum is exactly $\left\{s\om^{\bet} | \bet\in\bbZ_N\right\}$, since $(\pi_p\ot\pi_q)\Del(c_2)$ is invertible, $c_1c_2=\om^tc_2c_1$ and $t$ is invertible in $\bbZ_N$.
Futhermore, since $k_1c_1=c_1k_1$, $k_1c_2=c_2k_1$ and that $(\pi_p\ot\pi_q)\Del(c_1)$ and $(\pi_p\ot\pi_q)\Del(c_2)$ are invertible, we have 
\begin{align*}
(\pi_p\ot\pi_q)\Del(k_2)(W_{\alp})=(\pi_p\ot\pi_q)\Del(c_2)(W_{\alp})=W_{\alp}, \quad\forall \alp\in\bbZ_N,
\end{align*}
which implies that the spectrum of $(\pi_p\ot\pi_q)\Del(c_1)|_{_{W_{\alp}}}$ is $\left\{s\om^{\bet} | \bet\in\bbZ_N\right\}$. This allows us to define the eigenspaces $W_{(\alp,\bet)}$ for all $\alp,\bet\in\bbZ_N$ as announced. 

 Now, one can see that for all  $\alp,\bet\in\bbZ_N$, we have $\dim W_{(\alp,\bet)}=N^2$. Indeed, this comes from the fact that $\dim W_{\alp}=N^3$, $(\pi_p\ot\pi_q)\Del(c_2)|_{_{W_{\alp}}}$ is invertible, $c_1c_2=\om^tc_2c_1$ and $t$ is invertible in $\bbZ_N$.

Let $\left\{u_i\right\}_{i\in\bbZ_N^2}$ be a basis of $W_{0,0}$ and consider for all $\alp,\bet\in\bbZ_N$ and all $i\in\bbZ_N^2$
\begin{align*}
\xi_{(\alp,\bet),i}=\dfrac{1}{r^{\alp} s^{\bet}}(\pi_p\ot\pi_q)\Del(e_1^{\alp} c_2^{\bet})u_i\in W_{(\alp,\bet)},
\end{align*}
where $r=(p+q)^\frac{1}{N}\in\bbRz$.
By construction, we have for all $\alp,\bet\in\bbZ_N$ and all $i\in\bbZ_N^2$
\begin{align*}
(\pi_p\ot\pi_q)\Del(k_1) \xi_{(\alp,\bet),i} &=  \om^{-\alp}\xi_{(\alp,\bet),i},  \ \ &  (\pi_p\ot\pi_q)\Del(c_1) \xi_{(\alp,\bet),i} &=  s\om^{t\bet}\xi_{(\alp,\bet),i}, 
 \\ \\
 (\pi_p\ot\pi_q)\Del(e_1) \xi_{(\alp,\bet),i} &= r\xi_{(\alp+1,\bet),i}, \ \  &  (\pi_p\ot\pi_q)\Del(c_2) \xi_{(\alp,\bet),i} &= s\xi_{(\alp,\bet+1),i}.
\end{align*}
This clearly shows that for each $i\in\bbZ_N^2$, the subspace $\Xi_i\subset V_p\ot V_q$ generated by $\lac \xi_{(\alp,\bet),i}\rac_{\alp,\bet\in\bbZ_N}$ is an irreducible submodule. Furthermore, $\Xi_i$ is a reduced cyclic $\calA_{\om,t}$-module of parameter $p+q$. By Lemma \ref{equiv.rep}, the reduced cyclic $\calA_{\om,t}$-modules $\Xi_i$ are all equivalent.
\end{proof}

\begin{rem}
The dual space $V_p^*$ of a reduced cyclic $\calA_{\om,t}$-module $V_p$ can be provided with a $\calA_{\om,t}$-module structure $\pi_p^*: \calA_{\om,t}\to\End(V_p^*)$ defined, for all $a\in \calA_{\om,t}$ and all $\alp,\bet\in\bbZ_N^2$, by
\begin{align*}
\pa{\pi_p^*(a) v_\alp^*}v_\bet=v_\alp^*\pa{\pi_p(\gamma(a)) v_\bet}.
\end{align*}
This $\calA_{\om,t}$-module is actually equivalent to the reduced cyclic $\calA_{\om,t}$-module $V_{-p}$. 
\end{rem}



\section{A \texorpdfstring{$\Psi$}{\unichar{"03A8}}-system in the category of \texorpdfstring{$\calA_{\om,t}$}{A\_\{\unichar{"03C9},t\}}-modules}\label{sec.psi.sys.in.Awt}

%
Since $\calA_{\om,t}$ is a Hopf algebra over $\bbC$, the category of $\calA_{\om,t}$-modules, as a subcategory of the category of $\bbC$-vector spaces, is a monoidal abelian category. A $\Psi$-system in a monoidal abelian category is a family of simple object that satisfies special requirements. We are going to show that a $\Psi$-system in the category of $A_{\om,t}$-modules can be constructed with the family $\{V_p\}_{p\in\bbRz}$. We first recall the defintion of a $\Psi$-system in a monoidal category $\mathcal{C}$ with ground ring $\mathbb{K}$.
\begin{defin}[\cite{geer2012tetrahedral}]
A $\Psi$-system in a monoidal abelian category $\mathcal{C}$ 
consists of 
\begin{enumerate}
	
	\item a distinguished set of simple objects $\lac V_p\rac_{p\in \mathcal{P}}$ such that $\Hom_{\cal{C}}(V_p,V_q)=0$ for all $p \neq q$,
	
	\item an involution, $\mathcal{P}\to \mathcal{P}$, $p\mapsto p^*$
	
	\item two families of morphisms $\lac b_p: \mathbb{K}\to V_p\ot V_{p^*}\rac_{p\in \mathcal{P}}$ and $\lac d_p: V_p\ot V_{p^*}\to\mathbb{K}\rac_{p\in \mathcal{P}}$ such that for all $p\in \mathcal{P}$
		\begin{align}\label{del.bp.dp}
		(\Id_{V_p}\ot d_{p^*})(b_p\ot\Id_{V_p})=\Id_{V_p}=(d_p\ot\Id_{V_p})(\Id_{V_p}\ot b_{p^*})
		\end{align}
	
	\item Let $H_r^{p,q}$ and $H_{p,q}^r$ be $\Hom_{\mathcal{C}}(V_r,V_p\ot V_q)$ and $\Hom_{\mathcal{C}}(V_p\ot V_q,V_r)$ respectively. For any $p,q\in \mathcal{P}$ such that $H_r^{p,q}\neq 0$ for some $r\in \mathcal{P}$, the identity morphism $\Id_{V_p\ot V_q}$ is in the image of the linear map
	\begin{equation*}
	\begin{array}{ccc}
		 \dis\bigoplus_{r\in \mathcal{P}} H_r^{p,q}\ot H_{p,q}^r & \longrightarrow & \End_{\mathcal{C}}(V_p\ot V_q)
		 \\
  		 x\ot y & \longmapsto & x\circ y 
	\end{array}
	\end{equation*}
\end{enumerate}
\end{defin}
%
 From the previous section, the first point of the above definition is satisfied for the family $\{V_p\}_{p\in\bbRz}$. Moreover, using Proposition \ref{prod.tens} and Schur's Lemma, we have 
 \begin{enumerate}
     \item $H_r^{p,q}\neq 0$ if and only if $(p,q)\in\left(\bbRz\right)^2$ is admissible and $r = p+q$,
     \item $\dim H_{p+q}^{p,q} = \dim H^{p+q}_{p,q} = N^2$.
 \end{enumerate}
In \cite{geer2012tetrahedral}, when $H_r^{p,q}\neq 0$, the spaces $H_r^{p,q}$ and $H^r_{p,q}$ are called the {\it multiplicity spaces}. In our case, for any admissible pair $(p,q)$, we are going to simplify the notation as follows, 
\begin{align*}
    \mathcal{H}_{p,q} = H_{p+q}^{p,q} 
    \ \ \quad \text{and} \ \ \quad
    \bar{\mathcal{H}}_{p,q} = H^{p+q}_{p,q}. 
\end{align*}

In the following Subsections, we determine the morphisms $b_p$ and $d_p$ with the involution given by $p^*=-p$ and also determine bases for the multiplicity spaces $\mathcal{H}_{p,q}$ and $\bar{\mathcal{H}}_{p,q}$, 
using the $\mathfrak{sl}_3$ matrix dilogarithm defined in \cite[Exemple 5]{kashaev1999matrix}. With these results we will be able to show that $\lac V_p\rac_{p\in\bbRz}$ give rise to a $\Psi$-system in the category of $\calA_{\om,t}$-modules.

\subsection{Determination of the morphisms \texorpdfstring{$b_p$}{b\_\{p\}} and \texorpdfstring{$d_p$}{d\_\{p\}}}~\label{subsubsection:dp:bp}
A $\calA_{\om,t}$-module structure can be provided to $\bbC$ through the counit 
\begin{align*}
\eps:\calA_{\om,t}\to\End(\bbC)\cong\bbC.
\end{align*}
Let $p^* = -p$ be the involution on $\bbRz$ and consider the $\calA_{\om,t}$-modules $V_p$ and $V_{-p}$. The morphism $d_p$ is therefore an element of $\Hom_{\calA_{\om,t}}(V_p\ot V_{-p},\bbC)$. As such, $d_p$ satisfies 
\begin{align*}
\eps(a)d_p(u\ot v)& =d_p(\pi_p\ot\pi_{-p})\Del(a)(u\ot v)
\end{align*}
for all $a\in\calA_{\om,t}$ and all $u\ot v\in V_p\ot V_{-p}$. This equality will be used in the following Lemma.

\begin{lemma}\label{syst.d_p}
For any $p\in\bbRz$, the morphism $d_p$ is a solution of the following system of homogeneous linear equations
\begin{align*}
d_p &= d_p(X\ot X)\\
d_p &=d_p\left(XY^{-1}\ot Y\right)\\
d_p &=d_p\left(V^{-1}\ot X^{-(t+1)}U\right)\\
d_p &=\om^{-1}d_p\left(UV^{-1}\ot UV^{-1}\right).
\end{align*}
\end{lemma}
\begin{proof}
For all $a\in\calA_{\om,t}$ we have 
\begin{align}\label{eq.base.d_p}
\eps(a)d_p =d_p(\pi_p\ot\pi_{-p})\Del(a). 
\end{align}
If $a=k_1$, we have $\eps(k_1)=1$ and $(\pi_p\ot\pi_{-p})\Del(k_1)=X\ot X$. Hence, equality \eqref{eq.base.d_p} becomes
\begin{align}\label{d_p.de.X}
d_p=d_p(X\ot X).
\end{align} 
If $a=e_1$, we have $\eps(e_1)=0$ and 
\begin{align*}
(\pi_p\ot\pi_{-p})\Del(e_1)=p^\frac{1}{N}Y\ot \Id_{\calA} + (-p)^\frac{1}{N}X\ot Y.
\end{align*}
Hence, equality \eqref{eq.base.d_p} is equivalent to
\begin{align*}
0=d_p(Y\ot \Id_{\calA} - X\ot Y)  \ssi d_p(Y\ot\Id_{\calA})=d_p(X\ot Y) 
\end{align*}
which leads us to
\begin{align}\label{d_p.de.Y}
d_p=d_p(XY^{-1}\ot Y).
\end{align} 
If $a=e_2$, we have $\eps(e_2)=0$ and 
\begin{align*}
(\pi_p\ot\pi_{-p})\Del(e_2)=\pa{\tfrac{1}{2}p}^\frac{1}{N}Z\ot \Id_{\calA} + \pa{-\tfrac{1}{2}p}^\frac{1}{N}X^t\ot Z,
\end{align*}
where $Z=(U+V)Y^{-1}$. A similar computation to the previous one leads us to the following equality
\begin{align}\label{d_p.de.Z}
d_p=d_p(X^tZ^{-1}\ot Z)
\end{align}
If $a=a_1$, we have $\eps(a_1)=0$ and 
\begin{align*}
(\pi_p\ot\pi_{-p})\Del(a_1)=\pa{\tfrac{1}{2}p^2}^\frac{1}{N}U\ot \Id_{\calA} +\pa{\tfrac{1}{2}p^2}^\frac{1}{N}X^{t+1}\ot U + \pa{-\tfrac{1}{2}p^2}^\frac{1}{N}X^tY\ot Z,
\end{align*}
Therefore, using equality \eqref{eq.base.d_p} we compute
\begin{align*}
0&=d_p(U\ot \Id_{\calA} + X^{t+1}\ot U -X^tY\ot Z)
\\
&=d_p(U\ot \Id_{\calA})+d_p(X^{t+1}\ot U)-d_p(X^tY\ot Z)
\\
&\stackrel{\mathclap{\eqref{d_p.de.Z}}}{=}d_p(U\ot \Id_{\calA})+d_p(X^{t+1}\ot U)-d_p(ZY\ot\Id_{\calA})
\\
&=d_p(U\ot \Id_{\calA})+d_p(X^{t+1}\ot U)-d_p(U\ot\Id_{\calA})-d_p(V\ot\Id_{\calA})
\\
&=d_p(X^{t+1}\ot U)-d_p(V\ot\Id_{\calA})
\\
&\stackrel{\mathclap{\eqref{d_p.de.X}}}{=}d_p(\Id_{\calA}\ot X^{-(t+1)}U)-d_p(V\ot\Id_{\calA}),
\end{align*}
which leads to 
\begin{align}\label{d_p.de.U}
d_p=d_p(V^{-1}\ot X^{-(t+1)}U).
\end{align}
Finally, by reconsidering the last computation, we get
\begin{align*}
0&=d_p(U\ot \Id_{\calA} + X^{t+1}\ot U -X^tY\ot Z)
\\
&=d_p(U\ot \Id_{\calA})+d_p(X^{t+1}\ot U)-\om^{-t}d_p(YX^t\ot Z)
\\
&\stackrel{\mathclap{\eqref{d_p.de.Y}}}{=}d_p(U\ot \Id_{\calA})+d_p(X^{t+1}\ot U)-\om^{-t}d_p(X^{t+1}\ot YZ)
\\
&=d_p(U\ot \Id_{\calA})+d_p(X^{t+1}\ot U)-d_p(X^{t+1}\ot U)-\om d_p(X^{t+1}\ot V)
\\
&=d_p(U\ot \Id_{\calA})-\om d_p(X^{t+1}\ot V)
\\
&\stackrel{\mathclap{\eqref{d_p.de.X}}}{=} d_p(U\ot \Id_{\calA})-\om d_p(\Id_{\calA}\ot X^{-(t+1)}V),
\end{align*}
which leads to 
\begin{align*}
d_p=d_p(U^{-1}\ot X^{-(t+1)}V).
\end{align*}
From there, we use equality \eqref{d_p.de.U} to get
\begin{align*}
d_p(U^{-1}\ot X^{-(t+1)}V)=d_p(V^{-1}\ot X^{-(t+1)}U), 
\end{align*}
which is equivalent to
\begin{align*}
d_p=\om^{-1}d_p(UV^{-1}\ot UV^{-1}).
\end{align*}
\end{proof}

In order to take advantage of Lemma \ref{syst.d_p}, we are going to consider the following particular basis of $\calV$ instead of the canonical basis $\brkt{v_\alp}_{\alp\in\bbZ_N^2}$.
\begin{lemma}\label{bases.u.de.calV}
The set $\brkt{u_\alp}_{\alp\in\bbZ_N^2}\subset\calV$ where
\begin{align*}
u_\alp=\sum_{\bet\in\bbZ_N}\om^{-\frac{1}{2}t\bet(\bet+1)+\bet\pa{\alp_2-\alp_1+\frac{1}{2}}}v_{(\alp_1,\bet)}, 
\end{align*}
for all $\alp\in\bbZ_N^2$, is a basis of $\calV$. Its dual $\brkt{\bar{u}_\alp}_{\alp\in\bbZ_N^2}\subset\calV^*$ is given, for all $\alp\in\bbZ_N^2$, by
\begin{align*}
\bar{u}_\alp=\frac{1}{N}\sum_{\bet\in\bbZ_N}\om^{\frac{1}{2}t\bet(\bet+1)-\bet(\alp_2-\alp_1+\frac{1}{2})}\bar{v}_{(\alp_1,\bet)}. 
\end{align*}
\end{lemma}
\begin{proof}
For any $\alp\in\bbZ_N^2$ we have 
\begin{align*}
v_\alp=\frac{1}{N}\om^{\frac{1}{2}t\alp_2\pa{\alp_2+1}+\alp_2\pa{\alp_1-\frac{1}{2}}}\sum_{\bet\in\bbZ_N}\om^{-\alp_2\bet}u_{(\alp_1,\bet)}
\end{align*}
and 
\begin{align*}
\bar{v}_\alp=\om^{-\frac{1}{2}t\alp_2\pa{\alp_2+1}-\alp_2\pa{\alp_1-\frac{1}{2}}}\sum_{\bet\in\bbZ_N}\om^{\alp_2\bet}\bar{u}_{(\alp_1,\bet)}.
\end{align*}
This shows that $\brkt{u_\alp}_{\alp\in\bbZ_N^2}$, and $\brkt{\bar{u}_\alp}_{\alp\in\bbZ_N^2}$, are generating sets of $\calV$ and $\calV^*$ respectively. Hence these sets are bases since their cardinality is the same as the dimension of $\calV$ and $\calV^*$. 
A straightforward computation shows that for all $\alp,\bet\in\bbZ_N^2$ we have
\begin{align*}
\bar{u}_\bet u_\alp=\del_{\alp,\bet}
\end{align*}
which implies that $\brkt{u_\alp}_{\alp\in\bbZ_N^2}$, and $\brkt{\bar{u}_\alp}_{\alp\in\bbZ_N^2}$ are dual bases.
\end{proof}

The basis $\brkt{u_\alp}_{\alp\in\bbZ_N^2}$ satisfies the relations given in the following Lemma.

\begin{lemma}\label{Op.pour.u}
For all $\alp\in\bbZ_N^2$ and all $m\in\bbZ_N$
\begin{align*}
 X^m u_\alp & =\om^{-m\alp_1 }u_\alp, \ \  Y^m u_\alp= u_{(\alp_1+m,\alp_2+m)},
\\
U^m u_\alp &= \om^{-\frac{1}{2}m(m-1)(t+1)-mt\alp_1}u_{(\alp_1-mt,\alp_2)},
\\
V^m u_\alp & =\om^{-\frac{1}{2}m(m-1)(t+1)-m\pa{t\alp_1+\alp_2+\frac{1}{2}}}u_{(\alp_1-mt,\alp_2)}.
\end{align*}
In particular, we have for all $\alp\in\bbZ_N^2$ and all $m\in\bbZ_N$
\begin{align*}
(UV^{-1})^m u_\alp &= \om^{m(\alp_2+\frac{1}{2})} u_\alp, \\ (U^tV)^m u_\alp &= \om^{\frac{1}{2}m(m-1)+m\pa{\alp_1-\alp_2+\frac{1}{2}t}}u_{(\alp_1+m,\alp_2)}.
\end{align*}
\end{lemma}
\begin{proof}
For any $\alp\in\bbZ_N^2$ we have, by the equalities \eqref{def.el.calA}, 
\begin{align*}
 X v_\alp &= \om^{-\alp_1} v_\alp, \ \  Y v_\alp = v_{(\alp_1+1,\alp_2)},
\\
U v_\alp &= \om^{-t(\alp_1-\alp_2)}v_{(\alp_1-t,\alp_2)}, 
\\
V v_\alp &=\om^{-(t+1)\alp_1}v_{(\alp_1-t,\alp_2+1)}.
\end{align*}
Then we can easily derive the following equalities 
\begin{align*}
UV^{-1} v_\alp&= \om^{\alp_1+t\alp_2}v_{(\alp_1,\alp_2-1)}, 
\\ 
U^tV v_\alp&=\om^{-(t+1)(\alp_2+\frac{1}{2})}v_{(\alp_1+1,\alp_2+1)}.
\end{align*}

Now we determine the action of these operators on the basis $u_\alp$ for any $\alp\in\bbZ_N^2$. 
For $X$, we have 
\begin{align*}
Xu_\alp &=\sum_{\bet\in\bbZ_N}\om^{-\frac{1}{2}t\bet(\bet+1)+\bet(\alp_2-\alp_1+\frac{1}{2})}Xv_{(\alp_1,\bet)} 
\\
&=\sum_{\bet\in\bbZ_N}\om^{-\frac{1}{2}t\bet(\bet+1)+\bet(\alp_2-\alp_1+\frac{1}{2})}\om^{-\alp_1}v_{(\alp_1,\bet)}=\om^{-\alp_1}u_\alp.
\end{align*}
For $Y$, we have
\begin{align*}
Yu_\alp &=\sum_{\bet\in\bbZ_N}\om^{-\frac{1}{2}t\bet(\bet+1)+\bet(\alp_2-\alp_1+\frac{1}{2})}Yv_{(\alp_1,\bet)} 
\\
&=\sum_{\bet\in\bbZ_N}\om^{-\frac{1}{2}t\bet(\bet+1)+\bet(\alp_2-\alp_1+\frac{1}{2})}v_{(\alp_1+1,\bet)}
\\
&=\sum_{\bet\in\bbZ_N}\om^{-\frac{1}{2}t\bet(\bet+1)+\bet\pa{(\alp_2+1)-(\alp_1+1)+\frac{1}{2}}}v_{(\alp_1+1,\bet)}=u_{(\alp_1+1,\alp_2+1)}.
\end{align*}
For $U$, we have
\begin{align*}
Uu_\alp &=\sum_{\bet\in\bbZ_N}\om^{-\frac{1}{2}t\bet(\bet+1)+\bet(\alp_2-\alp_1+\frac{1}{2})}Uv_{(\alp_1,\bet)} 
\\
&=\sum_{\bet\in\bbZ_N}\om^{-\frac{1}{2}t\bet(\bet+1)+\bet(\alp_2-\alp_1+\frac{1}{2})}\om^{-t(\alp_1-\bet)}v_{(\alp_1-t,\bet)}
\\
&=\om^{-t\alp_1}\sum_{\bet\in\bbZ_N}\om^{-\frac{1}{2}t\bet(\bet+1)+\bet\pa{\alp_2-(\alp_1-t)+\frac{1}{2}}}v_{(\alp_1-t,\bet)}=\om^{-t\alp_1}u_{(\alp_1-t,\alp_2)}.
\end{align*}
For $V$, we have
\begin{align*}
Vu_\alp & =\sum_{\bet\in\bbZ_N}\om^{-\frac{1}{2}t\bet(\bet+1)+\bet(\alp_2-\alp_1+\frac{1}{2})}Vv_{(\alp_1,\bet)} 
\\
& =\sum_{\bet\in\bbZ_N}\om^{-\frac{1}{2}t\bet(\bet+1)+\bet(\alp_2-\alp_1+\frac{1}{2})}\om^{-(t+1)\alp_1}v_{(\alp_1-t,\bet+1)}
\\
& =\om^{-(t\alp_1+\alp_2+\frac{1}{2})}\sum_{\bet\in\bbZ_N}\om^{-\frac{1}{2}t(\bet+1)(\bet+2)+(\bet+1)\pa{\alp_2-(\alp_1-t)+\frac{1}{2}}}v_{(\alp_1-t,\bet+1)}
\\
& =\om^{-(t\alp_1+\alp_2+\frac{1}{2})}u_{(\alp_1-t,\alp_2)}.
\end{align*}
For $UV^{-1}$, we have
\begin{align*}
UV^{-1}u_\alp & =\sum_{\bet\in\bbZ_N}\om^{-\frac{1}{2}t\bet(\bet+1)+\bet(\alp_2-\alp_1+\frac{1}{2})}UV^{-1}v_{(\alp_1,\bet)} 
\\
& =\sum_{\bet\in\bbZ_N}\om^{-\frac{1}{2}t\bet(\bet+1)+\bet(\alp_2-\alp_1+\frac{1}{2})}\om^{\alp_1+t\bet}v_{(\alp_1,\bet-1)}
\\
& =\om^{\alp_2+\frac{1}{2}}\sum_{\bet\in\bbZ_N}\om^{-\frac{1}{2}t(\bet-1)\bet+(\bet-1)\pa{\alp_2-\alp_1+\frac{1}{2}}}v_{(\alp_1,\bet-1)}
=\om^{\alp_2+\frac{1}{2}}u_\alp.
\end{align*}
Finally, for $U^tV$, we have
\begin{align*}
U^tVu_\alp & =\sum_{\bet\in\bbZ_N}\om^{-\frac{1}{2}t\bet(\bet+1)+\bet\pa{\alp_2-\alp_1+\frac{1}{2}}}U^tVv_{(\alp_1,\bet)} 
\\
& =\sum_{\bet\in\bbZ_N}\om^{-\frac{1}{2}t\bet(\bet+1)+\bet\pa{\alp_2-\alp_1+\frac{1}{2}}}\om^{-(t+1)\pa{\bet+\frac{1}{2}}}v_{(\alp_1+1,\bet+1)}
\\
& =\om^{\alp_1-\alp_2+\frac{1}{2}t}\sum_{\bet\in\bbZ_N}\om^{-\frac{1}{2}t(\bet+1)(\bet+2)+(\bet+1)\pa{\alp_2-(\alp_1+1)+\frac{1}{2}}}v_{(\alp_1+1,\bet+1)}
\\
& =\om^{\alp_1-\alp_2+\frac{1}{2}t}u_{(\alp_1+1,\alp_2)}.
\end{align*}
The Lemma follows easily from the previous equalities.
\end{proof}
We can now determine the morphisms $b_p$ and $d_p$.
\begin{lemma}\label{det.b_p.d_p}
The morphisms $b_p: \bbC\to V_p\ot V_{-p}$ and $d_p: V_p\ot V_{-p}\to\bbC$ are  given by
\begin{align*}
b_p(1)=\sum_{\alp,\bet\in\bbZ_N^2} b_{p,\alp,\bet} u_\alp\ot u_\bet \\
\end{align*}
where
\begin{align}\label{formule.b}
b_{p,\alp,\bet}=\del_{\alp,-\bet}\om^{\frac{1}{2}\pa{(t+1)(\alp_1-\alp_2)(\alp_1-\alp_2+2)+\alp_1^2+\alp_2}},
\end{align}
and 
\begin{align}\label{formule.d}
d_p(u_\alp\ot u_\bet)=\del_{\alp,-\bet}\om^{-\frac{1}{2}\pa{(t+1)(\alp_1-\alp_2)(\alp_1-\alp_2-2)+\alp_1^2-\alp_2}}.
\end{align}
\end{lemma}

\begin{proof}
First we compute $d_p$ using Lemma \ref{syst.d_p} and Lemma \ref{Op.pour.u}. 
We start with the first and the last equality of the Lemma \ref{syst.d_p}. We have 
\begin{align*}
d_p(u_\alp\ot u_\bet)=d_p(X\ot X)(u_\alp\ot u_\bet)=\om^{-\alp_1-\bet_1}d_p(u_\alp\ot u_\bet)
\end{align*}
and 
\begin{align*}
d_p(u_\alp\ot u_\bet)=\om^{-1}d_p(UV^{-1}\ot UV^{-1})(u_\alp\ot u_\bet)=\om^{\alp_2+\bet_2}d_p(u_\alp\ot u_\bet).
\end{align*}
Hence we have 
\begin{align}
d_p(u_\alp\ot u_\bet)=\del_{\alp,-\bet}d_p(u_\alp\ot u_\bet).
\end{align}

We consider the third equality of the Lemma \ref{syst.d_p}
\begin{align*}
d_p=d_p(V^{-1}\ot X^{-(t+1)}U).
\end{align*}
Since $X,U,V\in\calA$ are invertible, this equality is equivalent to 
 \begin{align*}
d_p=d_p\pa{V^{-1}\ot X^{-(t+1)}U}^{t+1}
\end{align*}
A straightforward computation shows that
\begin{align*}
\pa{V^{-1}\ot X^{-(t+1)}U}^{t+1} = \om^{\frac{1}{2}}\pa{V^{-(t+1)}\ot X^{-t}U^{t+1}},
\end{align*}
hence, the following equality holds true
\begin{align*}
d_p= \om^{\frac{1}{2}} d_p\pa{V^{-(t+1)}\ot X^{-t}U^{t+1}}.
\end{align*}
Therefore, using Lemma \ref{Op.pour.u}, we compute 
\begin{align}\label{third.eq.d_p}
\begin{split}
d_p(u_\alp\ot u_\bet)&=\del_{\alp,-\bet}d_p(u_\alp\ot u_\bet)
\\
&=\del_{\alp,-\bet}\om^{\frac{1}{2}} d_p\pa{V^{-(t+1)}\ot X^{-t}U^{t+1}}(u_\alp\ot u_\bet)
\\
&=\del_{\alp,-\bet}\om^{\frac{1}{2}} d_p\pa{V^{-(t+1)}\ot X^{-t}U^{t+1}}(u_\alp\ot u_{-\alp})
\\
&=\del_{\alp,-\bet}\om^{-(t+2)\alp_1+(t+1)\alp_2+\frac{3}{2}t+2}d_p(u_{(\alp_1-1,\alp_2)}\ot u_{(-\alp_1+1,-\alp_2)})
\\
&=\del_{\alp,-\bet}\om^{-(t+2)\frac{1}{2}\alp_1(\alp_1+1)+\alp_1\pa{(t+1)\alp_2+\frac{3}{2}t+2}}d_p(u_{(0,\alp_2)}\ot u_{(0,-\alp_2)})
\\
&=\del_{\alp,-\bet}\om^{-\frac{1}{2}(t+2)\alp_1^2+(t+1)\alp_1\pa{\alp_2+1}}d_p(u_{(0,\alp_2)}\ot u_{(0,-\alp_2)}).
\end{split}
\end{align}

Now we use the second equality of Lemma \ref{syst.d_p} in the following way
\begin{align}\begin{split}
\label{second.eq.d_p}
d_p(u_{(0,\alp_2)}\ot u_{(0,-\alp_2)})&=d_p(XY^{-1}\ot Y)(u_{(0,\alp_2)}\ot u_{(0,-\alp_2)})
\\
&=\om d_p(u_{(-1,\alp_2-1)}\ot u_{(1,-\alp_2+1)})
\\
&\stackrel{\mathclap{\eqref{third.eq.d_p}}}{=}\om^{-(t+1)\alp_2-\frac{1}{2}t}d_p(u_{(0,\alp_2-1)}\ot u_{(0,-\alp_2+1)})
\\
&=\om^{-\frac{1}{2}(t+1)\alp_2(\alp_2+1)-\frac{1}{2}t\alp_2}d_p(u_{(0,0)}\ot u_{(0,0)})
\\
&=\om^{-\frac{1}{2}(t+1)\alp_2^2-\frac{1}{2}(2t+1)\alp_2}d_p(u_{(0,0)}\ot u_{(0,0)}).
\end{split}\end{align}
Using equalities \eqref{third.eq.d_p} and \eqref{second.eq.d_p}, we get 
\begin{align*}
d_p(u_\alp\ot u_\bet)&=\del_{\alp,-\bet}\om^{-\frac{1}{2}(t+2)\alp_1^2+(t+1)\alp_1\pa{\alp_2+1}}d_p(u_{(0,\alp_2)}\ot u_{(0,-\alp_2)})
\\
&=\del_{\alp,-\bet}\om^{-\frac{1}{2}(t+2)\alp_1^2+(t+1)\alp_1\pa{\alp_2+1}}\om^{-\frac{1}{2}(t+1)\alp_2^2-\frac{1}{2}(2t+1)\alp_2}d_p(u_{(0,0)}\ot u_{(0,0)})
\\
&=\del_{\alp,-\bet}\om^{-\frac{1}{2}\pa{(t+1)(\alp_1-\alp_2)(\alp_1-\alp_2-2)+\alp_1^2-\alp_2}}d_p(u_{(0,0)}\ot u_{(0,0)}).  
\end{align*}

Finally, we set $d_p(u_{(0,0)}\ot u_{(0,0)})=1$. The formula for $b_p$ is easily computed using formula (\ref{del.bp.dp}).
\end{proof}

\subsection{Two operator valued functions}
In order to construct bases for the multiplicity spaces, we need to introduce two operator valued functions. In the sequel, $\bbR\backslash\{0,1\}$ will be denoted by $\bbRzu$.
\smallskip

Let $\sf{U}\in\calA^{\ot 2}$ be an operator such that $\sfU^N=-\Id_{\calV^{\ot 2}}$ and $x\in\bbRzu$. The first function is the {\it quantum dilogarithm} $\Psi_{x}(\sfU)$ introduced by Faddeev and Kashaev \cite{faddeev1994quantum}. It is defined as a solution of the functional equation
\begin{align}\label{psi.fun.eqn}
\frac{\Psi_{x}(\om^{-1}\sfU)}{\Psi_{x}(\sfU)}=(1-x)^\frac{1}{N}-x^\frac{1}{N}\sfU.
\end{align}
Following \cite{geer2012tetrahedral}, we can write
\begin{align*}
\Psi_{x}(\sfU)=\sum_{\alp\in\bbZ_N}\psi_{x,\alp}(-\sfU)^\alp
\end{align*}
where $\psi_{x,\alp}\in\bbC$. For all $\alp\in\bbZ_N$, (\ref{psi.fun.eqn}) implies that 
\begin{align*}
\psi_{x,\alp}= \frac{x^\frac{1}{N}}{\om^{-\alp}-(1-x)^\frac{1}{N}} \psi_{x,\alp-1}.
\end{align*}
Then, for all $\alp\in\bbZ_N$, we write 
\begin{align}
\psi_{x,\alp}=\psi_{x,0}\prod_{j=1}^{\alp}\frac{x^\frac{1}{N}}{\om^{-j}-(1-x)^\frac{1}{N}},
\end{align}
where $\psi_{x,0}\in\bbC^*$ is chosen so that $\det \pa{\Psi_x(\sfU)}=1$. Using the notation of \cite{kashaev1993star}, we have 
\begin{align}\label{def.pti.psi}
\psi_{x,\alp}=\psi_{x,0}\om^{\frac{1}{2}\alp(\alp+1)}x^\frac{\alp}{N}w\pa{(1-x)^{\frac{1}{N}}\big|\alp}
\end{align}


where
\begin{align*}
w(x|\alp)=\prod_{j=1}^\alp\frac{1}{1-x\om^j}, 
\end{align*}
is defined for all $x\in\bbC$ such that $x^N\neq 1$ and all $\alp\in\brkt{0,\cdots,N-1}\subset\bbZ$.
The computation of the determinant of $\Psi_x(\sfU)^N$ is made possible due to the following formula shown in \cite[formula (2.4)]{karemera2016quantum}.
\begin{align}\label{Psi.to.N}
\Psi_x(\sfU)^N=\psi^N_{x,0}(1-x)^{1-N}\frac{D(1)}{D\pa{(1-x)^{-\frac{1}{N}}}D\pa{\sfU\om\pa{\tfrac{x}{1-x}}^{\frac{1}{N}}}}
\end{align}
where 
\begin{align}\label{def.D}
D(x)=\displaystyle{\prod_{j=1}^{N-1}}(1-x\om^j)^j .
\end{align}
\begin{lemma}\label{det.Psi}
For any $x\in\bbRzu$ and any operator $\sfU\in\calA^{\ot2}$ such that there exists an invertible opertator $\sfM\in\calA^{\ot2}$ such that $\sfU=-\sfM(X\ot\Id_\calV)\sfM^{-1}$, 
we have 
\begin{align}\label{def.psi0}
\psi_{x,0}^N=(1-x)^{\frac{N-1}{2}}D\pa{\pa{1-x}^{-\frac{1}{N}}}D(1)^{-1} \quad \rar \quad \det\pa{\Psi_x(\sfU)^N}=1. 
\end{align}
\end{lemma}
\begin{proof}
First we compute $\det\pa{D\pa{\sfX\om\pa{\tfrac{x}{x-1}}^{\frac{1}{N}}}}$ using the matrix form of $\sfX$ in the canonical basis  $\left\{v_\alp\right\}_{\alp\in\bbZ_N^2}$ of $\bbC^N$.
\begin{multline*}
 D\pa{\sfX\om\pa{\tfrac{x}{x-1}}^{\frac{1}{N}}}=\displaystyle{\prod_{j=1}^{N-1}}\pa{\Id_{\bbC^N}-\pa{\tfrac{x}{x-1}}^{\frac{1}{N}}\om^{j+1}\sfX}^j
\\
=
\displaystyle{\prod_{j=1}^{N-1}}
\begin{pmatrix}
 1-\pa{\tfrac{x}{x-1}}^{\frac{1}{N}}\om^{j+1} & & & \\
  & 1-\pa{\tfrac{x}{x-1}}^{\frac{1}{N}}\om^{j+2} & & \\
  & & \ddots & \\
  & & & 1-\pa{\tfrac{x}{x-1}}^{\frac{1}{N}}\om^{j+N}
\end{pmatrix}
^j
\end{multline*}
Hence we have 
\begin{align*}
\det\pa{D\pa{\sfX\om\pa{\tfrac{x}{x-1}}^{\frac{1}{N}}}} 
&=\displaystyle{\prod_{j=1}^{N-1}}\displaystyle{\prod_{i=1}^{N}}\pa{1-\pa{\tfrac{x}{x-1}}^{\frac{1}{N}}\om^{j+i}}^j 
\\
&=\displaystyle{\prod_{j=1}^{N-1}}\pa{1-\pa{\pa{\tfrac{x}{x-1}}^{\frac{1}{N}}\om^{j}}^N}^j
\\
&=\displaystyle{\prod_{j=1}^{N-1}}\pa{1-\pa{\tfrac{x}{x-1}}}^j
=\displaystyle{\prod_{j=1}^{N-1}}\pa{\tfrac{1}{1-x}}^j
=\pa{\tfrac{1}{1-x}}^{\frac{N(N-1)}{2}}
\end{align*}
Since $\sfU=-\sfM(X\ot\Id_\calV)\sfM^{-1}=-\sfM(\sfX\ot\Id_{\bbC^N}\ot\Id_\calV)\sfM^{-1}$, we have 
\begin{align*}
\det\pa{D\pa{\sfU\om\pa{\tfrac{x}{1-x}}^{\frac{1}{N}}}}=\det\pa{D\pa{\sfX\om\pa{\tfrac{x}{x-1}}^{\frac{1}{N}}}}^{N^3}=(1-x)^{-\frac{N^4(N-1)}{2}}.
\end{align*}
Therefore, we have 
\begin{align*}
\det\pa{\Psi_x(\sfU)^N}&=\det\pa{\psi^N_{x,0}(1-x)^{1-N}D(1)D\pa{(1-x)^{-\frac{1}{N}}}^{-1}D\pa{\sfU\om\pa{\tfrac{x}{1-x}}^{\frac{1}{N}}}^{-1}}   
\\
&=\pa{\psi^N_{x,0}(1-x)^{1-N}D(1)D\pa{(1-x)^{-\frac{1}{N}}}^{-1}(1-x)^{\frac{N-1}{2}}}^{N^4}
\\
&=\pa{\psi^N_{x,0}(1-x)^{\frac{1-N}{2}}D(1)D\pa{(1-x)^{-\frac{1}{N}}}^{-1}}^{N^4}.
\end{align*}
So if $\psi^N_{x,0}=(1-x)^{\frac{N-1}{2}}D\pa{\pa{1-x}^{-\frac{1}{N}}}D(1)^{-1}$ then $\det\pa{\Psi_x(\sfU)^N}=1$.
\end{proof}
\noindent
From now on, for all $x\in\bbRzu$, we assume that $\psi_{x,0}$ satisfies \eqref{def.psi0} and is such that $\det\pa{\Psi_x(\sfU)}=1$. 

The inverse of $\Psi_{x}(\sfU)$ that we will denote $\bar\Psi_{x}(\sfU)$ satisfies
\begin{align*}
\frac{\bar\Psi_{x}(\sfU)}{\bar\Psi_{x}(\om^{-1}\sfU)}=(1-x)^\frac{1}{N}-x^\frac{1}{N}\sfU \quad\quad \text{and} \quad\quad \Psi_{x}(\sfU)\bar\Psi_{x}(\sfU)=1.
\end{align*}
Setting
\begin{align*}
\bar\Psi_{x}(\sfU)=\sum_{\alp\in\bbZ_N}\bar\psi_{x,\alp}(-\sfU)^\alp
\end{align*}
we find, in a similar way, that for all $\alp\in\bbZ_N$
\begin{align}\label{def.pti.bar.psi}
\bar\psi_{x,\alp}=\bar{\psi}_{x,0}\pa{\tfrac{x}{x-1}}^\frac{\alp}{N}\om^\alp w\pa{(1-x)^{-\frac{1}{N}}\big|\alp}
\end{align}
where $\bar\psi_{x,0}\in\bbC^*$ is chosen according to $\psi_{x,0}\in\bbC^*$. In order to make this choice, we use following Lemma.
%
%

\begin{lemma}[\cite{karemera2016quantum}, Lemma 2.3]\label{det.bar.Psi}
There exists $a\in\bbZ_N$ such that for any $x\in\bbRzu$ and any operator $\sfU\in\calA^{\ot2}$ such that $\sfU^N=-\Id_{\calV^{\ot2}}$ we have 
\begin{align}\label{def.bar.psi0}
\bar{\psi}_{x,0} =\om^a \ro\pa{\pa{1-x}^{\frac{1}{N}}}\psi_{x,0}^{-1} \quad \rar \quad \Psi_{x}(\sfU)\bar\Psi_{x}(\sfU)=1 
\end{align}
where 
\begin{align*}
\ro(x)=N^{-1}\frac{1-x^N}{1-x}.
\end{align*}
\end{lemma}

\noindent
From now on, for all $x\in\bbRzu$, we assume that $\bar{\psi}_{x,0}$ satisfies \eqref{def.bar.psi0}.

\smallskip
The second operator valued function, is defined by 
\begin{align}\label{def.L.operator}
L(\sfU,\sfV)=\frac{1}{N}\sum_{\alp,\bet\in\bbZ_N}\om^{-\alp\bet}\sfU^\alp\ot \sfV^\bet.
\end{align}
where $\sfU,\sfV\in\calA$ satisfy $\sfU^N=\sfV^N=\Id_\calV$. We derive two properties in the following Lemmas.
\begin{lemma}\label{prop.L1}
For any $\alp,\bet\in\bbZ_N$ and any $\sfU,\sfV\in\calA$ such that $\sfU^N=\sfV^N=\Id_\calV$, we have 
\begin{align*}
 L(\om^\alp \sfU,\om^\bet \sfV)
 =L(\sfU,\sfV)\om^{\alp\bet}\sfU^\bet\ot \sfV^\alp.
\end{align*} 
\end{lemma}
\begin{proof}
We compute
\begin{align*}
L(\om^\alp \sfU,\om^\bet \sfV) &=  \frac{1}{N}\sum_{\gamma,\del\in\bbZ_N}\om^{-\gamma\del+\alp\gamma+\bet\del}\sfU^\gamma\ot \sfV^\del
\\
&=\frac{1}{N}\sum_{\gamma,\del\in\bbZ_N}\om^{-\gamma(\del-\alp)+\bet\del}\sfU^\gamma\ot \sfV^{\del-\alp+\alp}
\\
&=  \frac{1}{N}\sum_{\gamma,\del\in\bbZ_N}\om^{-\gamma\del+\bet(\del+\alp)}\sfU^\gamma\ot \sfV^{\del}\sfV^\alp
\\
&=\frac{1}{N}\sum_{\gamma,\del\in\bbZ_N}\om^{-(\gamma-\bet)\del+\alp\bet}\sfU^{\gamma-\bet+\bet}\ot \sfV^{\del}\sfV^\alp
\\
&=  \frac{1}{N}\sum_{\gamma,\del\in\bbZ_N}\om^{-\gamma\del}\sfU^{\gamma}\ot \sfV^{\del}\om^{\alp\bet} \sfU^\bet\ot \sfV^\alp
\\
&=L(\sfU,\sfV)\om^{\alp\bet}\sfU^\bet\ot \sfV^\alp.
\end{align*}
\end{proof}

\begin{lemma}\label{det.L}
For any $\sfU\in\calA$ such that $\sfU^N=\Id_\calV$, we have 
\begin{align}
\det\pa{L(\sfU,X)}=1.
\end{align}
\end{lemma}
\begin{proof}
We consider the operators of $\calA$ in their matrix form in the basis $\brkt{v_\alp}_{\alp\in\bbZ_N^2}$ of $\calV$. For any $\bet\in\bbZ_N$, we compute
\begin{align*}
X^\bet=\sfX^\bet\ot\Id_{\bbC^N}
=\begin{pmatrix}
 \framebox{$\Id_{\bbC^N}$} & & & \\
  & \framebox{$\om^\bet\Id_{\bbC^N}$} & & \\
  & & \ddots & \\
  & & & \framebox{$\om^{\bet(N-1)}\Id_{\bbC^N}$} 
\end{pmatrix}.
\end{align*}
Thereby, for any $\alp\in\bbZ_N$, we have
\begin{align*}
\sum_{\bet\in\bbZ_N}\om^{-\alp\bet}X^\bet
 &=\sum_{\bet\in\bbZ_N}
\begin{pmatrix}
 \framebox{$\om^{-\alp\bet}\Id_{\bbC^N}$} & & & \\
  & \framebox{$\om^{\bet(1-\alp)}\Id_{\bbC^N}$} & & \\
  & & \ddots & \\
  & & & \framebox{$\om^{\bet(N-1-\alp)}\Id_{\bbC^N}$} 
\end{pmatrix}
\\
&=N\begin{pmatrix}
 \framebox{$\del_{\alp,0}\Id_{\bbC^N}$} & & & \\
  & \framebox{$\del_{\alp,1}\Id_{\bbC^N}$} & & \\
  & & \ddots & \\
  & & & \framebox{$\del_{\alp,N-1}\Id_{\bbC^N}$} 
\end{pmatrix}.
\end{align*}
Let $\sfM\in\calA^{\ot2}$ be the permutation matrix such that for all $\sfV,\sfW\in\calA$ we have 
$$\sfW\ot\sfV=\sfM\sfV\ot\sfW\sfM^{-1}.$$
Then we compute 
\begin{align*}
\sfM L(\sfU,X)\sfM^{-1}
\!&\!=\frac{1}{N}\sum_{\alp,\bet\in\bbZ_N}\om^{-\alp\bet}X^\bet\ot\sfU^\alp 
=\sum_{\alp\in\bbZ_N}\pa{\frac{1}{N}\sum_{\bet\in\bbZ_N}\om^{-\alp\bet}X^\bet\ot\sfU^\alp}
\\
&=\!\!\sum_{\alp\in\bbZ_N}\!\!
\begin{pmatrix}
 \framebox{$\del_{\alp,0}\Id_{\bbC^N}\ot\sfU^\alp$} & & & \\
  & \!\!\framebox{$\del_{\alp,1}\Id_{\bbC^N}\ot\sfU^\alp$} & & \\
  & & \!\!\ddots & \\
  & & & \!\!\framebox{$\del_{\alp,N-1}\Id_{\bbC^N}\ot\sfU^\alp$} 
\end{pmatrix}
\\
&=
\begin{pmatrix}
 \framebox{$\Id_{\bbC^N}\ot\Id_{\bbC^N}$} & & & \\
  & \framebox{$\Id_{\bbC^N}\ot\sfU$} & & \\
  & & \ddots & \\
  & & & \framebox{$\Id_{\bbC^N}\ot\sfU^{N-1}$} 
\end{pmatrix}.
\end{align*}
Finally we have 
\begin{align*}
\det\pa{L(\sfU,X)}=\det\pa{\sfM L(\sfU,X)\sfM^{-1}}=\det\pa{\sfU^{\frac{N(N-1)}{2}}}=1.
\end{align*}
\end{proof}

\subsection{Matrix dilogarithm and bases for the multiplicity spaces}
%

Using the fonctions we just defined, we are going to define an operator $S(x)$, where $x\in\bbRzu$. This operator, which correspond to the $\mathfrak{sl}_3$ matrix dilogarithm defined by Kashaev in \cite[Exemple 5]{kashaev1999matrix}, plays a key role as it  will allow us to construct bases for the multiplicity spaces as well as the $6j$-symbols.
\begin{notation}
Let $\sfU$ and $\sfV$ be two operators such that $\sfU\sfV=\om^\alp \sfV\sfU$, where $\alp\in\bbZ_N$. We will write this relation in the following way 
\begin{center}
\includegraphics[width=1\linewidth]{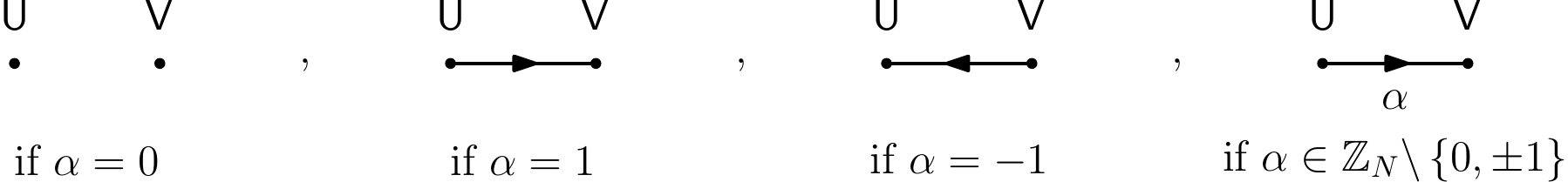}
\end{center} 
\end{notation}
\noindent
Note that in the case of operators $X,Y,U$ and $V$, we have the following relations
\begin{center}
\includegraphics[width=0.28\linewidth]{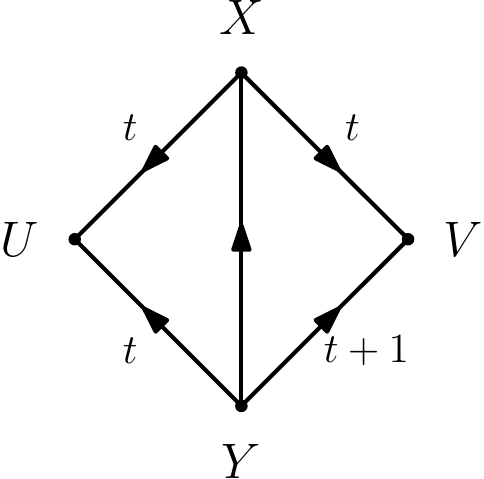}
\end{center}

We now consider, for any $x\in\bbRzu$, the following invertible operator valued function, which is the matrix dilogarithm given in \cite[Exemple 5]{kashaev1999matrix}, 
\begin{align*}
S(x)=\Psi_{x}(E)\Psi_{x}(F)\Psi_{x}(G)\Psi_{x}(H)L(U^tV,X)
\end{align*}
where
\begin{align}\label{def.EFGH}
E=-Y_1^{-1}X_1Y_2, \ \ F=U_1^{-1}X_1^{t+1}V_2E^{-1}, \ \ G=U_2V_2^{-1}F, \ \ H=U_1V_1^{-1}E.
\end{align}
and the subscripts show how the operators are embedded in $\calA^{\ot2}$.  These operators have the following commutation relations 
\begin{center}
\includegraphics[width=0.25\linewidth]{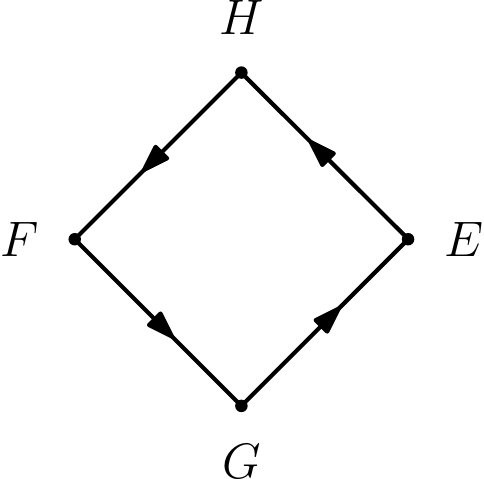}
\end{center}
Note that $E,F,G$ and $H$ satisfy the condition of Lemma \ref{det.bar.Psi}. Thereby, using Lemma \ref{det.L}, we see that the operator $S(x)$ is invertible since $\det\pa{S(x)^N}=1$.
Moreover, since
\begin{align*}
L(U^tV,X)^{-1}=L(U^{-t}V^{-1},X),
\end{align*} 
we have
\begin{align*}
S(x)^{-1}=L(U^{-t}V^{-1},X)\bar\Psi_{x}(H)\bar\Psi_{x}(G)\bar\Psi_{x}(F)\bar\Psi_{x}(E).
\end{align*}

\begin{rem}
The quantum dilogarithm of Fadeev and Kashaev defined in \cite{faddeev1994quantum} correspond to an $\mathfrak{sl}_2$ analogue of the operator $S(x)$ (see  Exemple 4 in \cite{kashaev1999matrix}). Kashaev's invariants as well as Baseilhac and Benedetti's invariants are based on the latter.
\end{rem}


\begin{proposition}\label{prop.base.Hpq}
For any admissible pair $(p,q)\in(\bbRz)^2$, the following equation is satisfied 
\begin{align}\label{S.act.commut}
(\pi_p\ot\pi_q)\Del(a)=S\left(\tfrac{q}{p+q}\right)(\pi_{p+q}(a)\ot\Id_{\calA})S\left(\tfrac{q}{p+q}\right)^{-1}, \ \ \forall a\in A_{\om,t}
\end{align}
\end{proposition}
\begin{proof}
Since $e_2e_1=a_1+a_2$, it is enough to check equation \eqref{S.act.commut} for $a\in\lac k_1,e_1, a_1, a_2\rac$. In the following computaions, we set $x=\tfrac{q}{p+q}\in\bbRzu$ and we will thus write $\frac{q-xq}{x}$ instead of $p$ and $\frac{q}{x}$ instead of $p+q$. 
\vspace{0.1cm}

\underline{For $a=k_1$:} Since $X U^tV= \om^{-1}U^tVX$, we have, by Lemma \ref{prop.L1},
\begin{align*}
	L(U^tV,X)X_1 &=\frac{1}{N}\sum_{\alp,\bet\bbZ_N}\om^{-\alp\bet}X_1(\om U_1^tV_1)^\alp X_2^\bet 
	\\
	&=X_1L(\om U^tV,X)
	=X_1X_2L(U^tV,X).
\end{align*}
Hence, since $X_1X_2$ commutes with $E,F,G$ and $H$, we have, using \eqref{act.Awt}
\begin{align*}
	S(x)(\pi_{\frac{q}{x}}(k_1)\ot\Id_{\calA}) &= \Psi_{x}(E)\Psi_{x}(F)\Psi_{x}(G)\Psi_{x}(H)L(U^tV,X)X_1 
	\\
	 &= \Psi_{x}(E)\Psi_{x}(F)\Psi_{x}(G)\Psi_{x}(H)X_1X_2L(U^tV,X)
	 \\
	 &= X_1X_2\Psi_{x}(E)\Psi_{x}(F)\Psi_{x}(G)\Psi_{x}(H)L(U^tV,X) 
	 \\
	 &= \left(\pi_{\frac{q-xq}{x}}\ot\pi_q\right)\Del(k_1)S(x).
\end{align*}
	
\underline{For $a=e_1$:} Since $Y_1$ commutes with $F,G, H$ and $U_1^tV_1$ and that $Y_1E=\om EY_1$ we have, using \eqref{act.Awt} and equation \eqref{psi.fun.eqn}
\begin{align*}
	S(x)(\pi_{\frac{q}{x}}(e_1)\ot\Id_{\calA}) &= \Psi_{x}(E)\Psi_{x}(F)\Psi_{x}(G)\Psi_{x}(H)L(U^tV,X)\pa{\tfrac{q}{x}}^{\frac{1}{N}}Y_1 
	\\
	 &= \pa{\tfrac{q}{x}}^{\frac{1}{N}}Y_1\Psi_{x}(\om^{-1}E)\Psi_{x}(F)\Psi_{x}(G)\Psi_{x}(H)L(U^tV,X) 
	\\
	 &= \pa{\tfrac{q}{x}}^{\frac{1}{N}}Y_1\left(\pa{1-x}^{\frac{1}{N}}-x^\frac{1}{N}E\right)S(x) 
	 \\
	 &= \left(\pa{\tfrac{q-xq}{x}}^\frac{1}{N}Y_1+q^\frac{1}{N}X_1Y_2\right)S(x)
	 \\
	 &= \left(\pi_{\frac{q-xq}{x}}\ot\pi_q\right)\Del(e_1)S(x).
\end{align*}

\underline{For $a=a_1$:} Since $U_1$ commutes with $U_1^tV_1$ and has the following commutation relations 
\begin{center}
\includegraphics[width=0.25\linewidth]{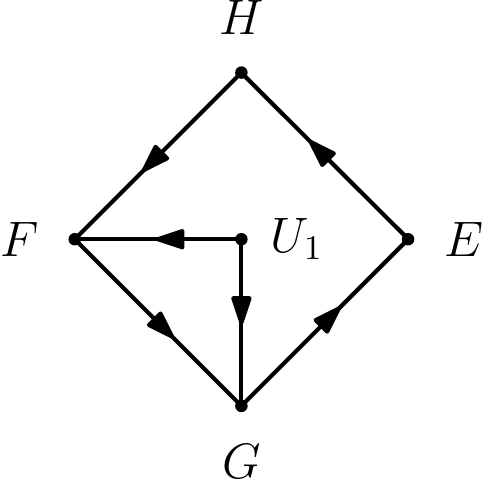}
\end{center}
we have, 
\begin{align*}
	S(x)(\pi_{\frac{q}{x}}(a_1)\ot\Id_{\calA})&= \Psi_{x}(E)\Psi_{x}(F)\Psi_{x}(G)\Psi_{x}(H)L(U^tV,X)\left(\tfrac{q^2}{2x^2}\right)^{\frac{1}{N}}U_1 
	\\
	&= \left(\tfrac{q^2}{2x^2}\right)^{\frac{1}{N}}U_1\Psi_{x}(E)\Psi_{x}(\om^{-1}F)\Psi_{x}(\om^{-1}G)\Psi_{x}(H)L(U^tV,X).
\end{align*}
Now, using equality \eqref{psi.fun.eqn}, we compute
\begin{align*}
	&\Psi_{x}(E)\Psi_{x}(\om^{-1}F)\Psi_{x}(\om^{-1}G)=\Psi_{x}(E)\Psi_{x}(\om^{-1}F)\left((1-x)^{\frac{1}{N}}-x^\frac{1}{N}G\right)\Psi_{x}(G)
	\\
	&=\Psi_{x}(E)\left((1-x)^{\frac{1}{N}}\Psi_{x}(\om^{-1}F)-x^\frac{1}{N}G\Psi_{x}(F)\right)\Psi_{x}(G)
	\\
	&=\left((1-x)^{\frac{1}{N}}\left((1-x)^{\frac{1}{N}}-x^\frac{1}{N}F\right)\Psi_{x}(E)-x^\frac{1}{N}G\Psi_{x}(\om^{-1}E)\right) \Psi_{x}(F)\Psi_{x}(G)
	\\
	&=\left((1-x)^\frac{2}{N}-(x-x^2)^\frac{1}{N}F-x^\frac{1}{N}G\left((1+x)^{\frac{1}{N}}-x^\frac{1}{N}E\right)\right)\Psi_{x}(E)\Psi_{x}(F)\Psi_{x}(G)
	\\
	&=\left((1-x)^\frac{2}{N}+x^\frac{2}{N}GE-(x-x^2)^\frac{1}{N}(F+G)\right)\Psi_{x}(E)\Psi_{x}(F)\Psi_{x}(G).
\end{align*}
Therefore, we finally have, using \eqref{def.EFGH} and \eqref{pi.a}
\begin{align*}
	&S(x)(\pi_{\frac{q}{x}}(a_1)\ot\Id_\calV) 
	= \left(\tfrac{q^2}{2x^2}\right)^{\frac{1}{N}}U_1\Psi_{x}(E)\Psi_{x}(\om^{-1}F)\Psi_{x}(\om^{-1}G)\Psi_{x}(H)L(U^tV,X)
	\\
	&=\left(\tfrac{q^2}{2x^2}\right)^{\frac{1}{N}}U_1\left((1-x)^\frac{2}{N}+x^\frac{2}{N}GE-(x-x^2)^\frac{1}{N}(F+G)\right)S(x)
	\\
	& =\left(\left(\tfrac{1}{2}(\tfrac{q-xq}{x})^2\right)^\frac{1}{N}U_1+\left(\tfrac{1}{2}q^2\right)^\frac{1}{N}U_1GE-\left(\tfrac{q^2-xq^2}{2x}\right)^\frac{1}{N}U_1(F+G)\right)S(x) 
	 \\
	& = \left(\left(\tfrac{1}{2}(\tfrac{q-xq}{x})^2\right)^\frac{1}{N}U_1+\left(\tfrac{1}{2}q^2\right)^\frac{1}{N}X_1^{t+1}U_2-\left(\tfrac{q^2-xq^2}{2x}\right)^\frac{1}{N}X_1^tY_1(U_2+V_2)Y_2^{-1}\right)S(x) 
	 \\
	& = \left(\pi_{\frac{q-xq}{x}}\ot \pi_q\right)\left( a_1\ot 1 + k_1^{t+1}\ot a_1 + k_1^te_1\ot e_2 \right)S(x) 
	\\
	&= \left(\pi_{\frac{q-xq}{x}}\ot\pi_q\right)\Del(a_1)S(x).
\end{align*}

\underline{For $a=a_2$:} Since $V_1$ commutes with $U_1^tV_1$ and has the following commutation relations %

\begin{center}
\includegraphics[width=0.25\linewidth]{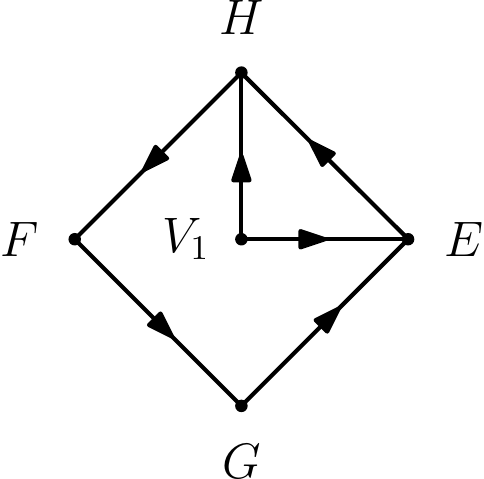}
\end{center}
we have, 
\begin{align*}
	S(x)(\pi_{\frac{q}{x}}(a_2)\ot\Id_\calV) &= \Psi_{x}(E)\Psi_{x}(F)\Psi_{x}(G)\Psi_{x}(H)L(U^tV,X)\left(\tfrac{q^2}{2x^2}\right)^{\frac{1}{N}}V_1 
	\\
	& = \left(\tfrac{q^2}{2x^2}\right)^{\frac{1}{N}}V_1\Psi_{x}(\om^{-1}E)\Psi_{x}(F)\Psi_{x}(G)\Psi_{x}(\om^{-1}H)L(U^tV,X) 
	\\
	& =\left(\tfrac{q^2}{2x^2}\right)^{\frac{1}{N}}V_1\Psi_{x}(\om^{-1}E)\Psi_{x}(F)\Psi_{x}(\om^{-1}H)\Psi_{x}(G)L(U^tV,X).
	\end{align*}
Now, using equality \eqref{psi.fun.eqn}, we compute 
\begin{align*}
\Psi_{x}(F)\Psi_{x}(\om^{-1}H) &=\Psi_{x}(F)\left((1-x)^\frac{1}{N}-x^\frac{1}{N}H\right)\Psi_{x}(H) 
	\\
&	=\left((1-x)^\frac{1}{N}\Psi_{x}(F)-x^\frac{1}{N}H\Psi_{x}(\om^{-1}F)\right)\Psi_{x}(H) 
	\\
&	=\left((1-x)^\frac{1}{N}\Psi_{x}(F)-x^\frac{1}{N}H\left((1-x)^\frac{1}{N}-x^\frac{1}{N}F\right)\Psi_{x}(F)\right)\Psi_{x}(H) 
	\\
&	=\left((1-x)^\frac{1}{N}+x^\frac{2}{N}HF-(x-x^2)^\frac{1}{N}H\right)\Psi_{x}(F)\Psi_{x}(H).
	\end{align*}
Hence, using equality \eqref{psi.fun.eqn} again and the previous equality, we have 
\begin{align*}
&\Psi_{x}(\om^{-1}E)\Psi_{x}(F)\Psi_{x}(\om^{-1}H) 
\\
&=\Psi_{x}(\om^{-1}E)\left((1-x)^\frac{1}{N}+x^\frac{2}{N}HF-(x-x^2)^\frac{1}{N}H\right)\Psi_{x}(F)\Psi_{x}(H).
\\
&=\left((1-x)^\frac{1}{N}\left((1-x)^\frac{1}{N}-x^\frac{1}{N}E\right)+x^\frac{2}{N}HF-(x-x^2)^\frac{1}{N}H\right)\Psi_{x}(E)\Psi_{x}(F)\Psi_{x}(H) 
\\
&=\left((1-x)^\frac{2}{N}+ x^\frac{2}{N}HF-(x-x^2)^\frac{1}{N}(E+H)\right)\Psi_{x}(E)\Psi_{x}(F)\Psi_{x}(H).
\end{align*}
Therefore, using \eqref{def.EFGH}, \eqref{pi.a} and the fact that $\Psi_{x}(H) \Psi_{x}(G)=\Psi_{x}(G) \Psi_{x}(H)$, we finally have
\begin{align*}
	&S(x)(\pi_{\frac{q}{x}}(a_2)\ot\Id_\calV) 
	 =\left(\tfrac{q^2}{2x^2}\right)^{\frac{1}{N}}V_1\Psi_{x}(\om^{-1}E)\Psi_{x}(F)\Psi_{x}(\om^{-1}H)\Psi_{x}(G)L(U^tV,X)
		 \\
		& =\left(\tfrac{q^2}{2x^2}\right)^{\frac{1}{N}}V_1\left((1-x)^\frac{2}{N}+ x^\frac{2}{N}HF-(x-x^2)^\frac{1}{N}(E+H)\right)S(x) 
		 \\
		& =\left(\left(\tfrac{1}{2}(\tfrac{q-xq}{x})^2\right)^\frac{1}{N}V_1+ \left(\tfrac{1}{2}q^2\right)^\frac{1}{N}V_1HF-\left(\tfrac{q^2-xq^2}{2x}\right)^\frac{1}{N}V_1(E+H)\right) S(x) 
		 \\
		& =\left(\left(\tfrac{1}{2}(\tfrac{q-xq}{x})^2\right)^\frac{1}{N}V_1+ \left(\tfrac{1}{2}q^2\right)^\frac{1}{N}X^{t+1}_1V_2-\left(\tfrac{q^2-xq^2}{2x}\right)^\frac{1}{N}(U_1+V_1)Y_1^{-1}X_1Y_2\right) S(x) 
		 \\
		& = \left(\pi_{\frac{q-xq}{x}}\ot \pi_q\right)\left( a_2\ot 1 + k_1^{t+1}\ot a_2 + e_2k_1\ot e_1 \right)S(x) 
		\\
		& = \left(\pi_{\frac{q-xq}{x}}\ot \pi_q\right)\Del(a_2)S(x).
\end{align*}
\end{proof}

We can now define bases for the multiplicity spaces. For any $x\in\bbRzu$ and $\alp\in\bbZ^2_N$, we define 
\begin{align}\label{def.bas.H}
e_\alp(x)=S(x)(\Id_\calV\ot u_\alp) \ \ \text{and} \ \ \bar{e}_\alp(x)=(\Id_\calV\ot \bar{u}_\alp)S(x)^{-1}.
\end{align}
If $(p,q)\in(\bbRz)^2$ is an admissible pair, then by equation \eqref{S.act.commut}, 
\begin{align*}
e_\alp\pa{\tfrac{q}{p+q}}\in\calH_{p,q}\quad \text{and} \quad\bar{e}_\alp\pa{\tfrac{q}{p+q}}\in\bar{\calH}_{p,q},\quad \forall\alp\in\bbZ^2_N.
\end{align*} 
Moreover, $\brkt{e_{\alp}\pa{\tfrac{q}{p+q}}}_{\alp\in\bbZ_N^2}$ and $\brkt{ \bar{e}_\alp\pa{\tfrac{q}{p+q}}}_{\alp\in\bbZ_N^2}$ form dual basis of $\calH_{p,q}$ and $\bar{\calH}_{p,q}$ respectively, where the duality is reflected by the relations
\begin{align}
\bar{e}_\bet\pa{\tfrac{q}{p+q}} e_\alp\pa{\tfrac{q}{p+q}}=\del_{\alp,\bet}\Id_\calV, 
\end{align}
and
\begin{align}\label{psi.cond.4}
\sum_{\alp\in\bbZ_N^2}e_\alp\pa{\tfrac{q}{p+q}} \bar{e}_\alp\pa{\tfrac{q}{p+q}}=\Id_\calV\ot\Id_\calV
\end{align}

\begin{rem}
$\mathfrak{sl}_2$ analogue of bases $\brkt{e_{\alp}\pa{x}}_{\alp\in\bbZ_N^2}$ and $\brkt{ \bar{e}_\alp\pa{x}}_{\alp\in\bbZ_N^2}$ are defined in \cite[Section 12]{geer2012tetrahedral} using the associated operator $S(x)$ similarly as in \eqref{def.bas.H}.  
\end{rem}

As we have just shown, the non-trivial elements of the multiplicity spaces $\calH_{p,q}$ and $\bar{\calH}_{p,q}$ depend only on $\tfrac{q}{p+q}\in\bbRzu$. We will call this property {\it the scaling invariance property} of the multiplicity spaces. 
In order to keep track of the difference between $\calH_{p,q}$ and $\calH_{\lambda p, \lambda q}$ and the difference between $\bar{\calH}_{p,q}$ and $\bar{\calH}_{\lambda p, \lambda q}$ respectively, where $\lambda\in\bbRz$, we define, for all admissible pair $(p,q)\in(\bbRz)^2$, the following isomorphism 
\begin{align*}
h_{p,q}:\calV\to\calH_{p,q}\quad \text{and} \quad \bar{h}_{p,q}:\calV^*\to\bar{\calH}_{p,q}
\end{align*}
by 
\begin{align*}
h_{p,q}(u_\alp)=e_\alp\pa{\tfrac{q}{p+q}},\quad  \quad \bar{h}_{p,q}(\bar{u}_\alp)=\bar{e}_\alp\pa{\tfrac{q}{p+q}}, \quad\forall \alp\in\bbZ^2_N.
\end{align*}

\subsection{The \texorpdfstring{$\Psi$}{\unichar{"03A8}}-system}
%

\begin{theorem}\label{thm.Psi.syst}
In the category of $\calA_{\om,t}$-modules, the set of objects $\lac V_p\rac_{p\in\bbRz}$ with the involution $p^*=-p$ and the duality morphisms defined in Lemma \ref{det.b_p.d_p} is a $\Psi$-system. 
\end{theorem}

\begin{proof}
By definition,  we have to check the  following three points :  
\begin{enumerate}	
	\item  $\Hom(V_p,V_q)=0$ for all $p \neq q$,
	\item The morphisms $\lac b_p: \bbC\to V_p\ot V_{-p}\rac_{p\in\bbRz}$ and $\lac d_p: V_p\ot V_{-p}\to\bbC\rac_{p\in\bbRz}$ satisfy
		\begin{align}
		(\Id_\calV\ot d_{-p})(b_p\ot\Id_\calV)=\Id_\calV=(d_p\ot\Id_\calV)(\Id_\calV\ot b_{-p}),  \ \ \forall p\in\bbRz
		\end{align}
	
	\item If $(p,q)$ is admissible, then $\Id_{V_p\ot V_q}$ is in the image of the linear map
		\begin{align*}
		\calH_{p,q}\ot \bar{\calH}_{p,q}\to \End (V_p\ot V_q),\ \ x\ot y\mapsto x y.
		\end{align*}
\end{enumerate}
Point (1) is clear by Schur's Lemma, point (2) is straightforward using Lemma \ref{det.b_p.d_p} and point (3) is given by formula (\ref{psi.cond.4}). 
\end{proof}




\section[Operators in the space of multiplicities]{Operators in the space of multiplicities}\label{sec.op.in.H}

We consider the vector space $\calH$, called the {\it space of multiplicities} in \cite{geer2012tetrahedral}, defined as $\calH = \check{\mathcal{H}} \oplus \hat{\mathcal{H}}$ where
\begin{align*}
\check{\calH}=\bigoplus_{\substack{(p,q)\in(\bbRz)^2 \\ \text{admissible}}}\calH_{p,q} \ \ \ \ \text{and} \ \ \ \ \hat{\calH}=\bigoplus_{\substack{(p,q)\in(\bbRz)^2 \\ \text{admissible}}}\bar{\calH}_{p,q}.
\end{align*}
We are going to determine the key operators in $\End(\calH)$ that will allow us to extend the $\Psi$-system defined in Theorem \ref{thm.Psi.syst} into a $\hat{\Psi}$-system. Before we do so, we give some defintions.

Let us first recall that the {\it Mobius group} is defined as the group $\PGL(2,\bbC)$ acting on $\bbC\cup\brkt{\infty}$ as follows
\begin{align*}
\begin{pmatrix}
a & b
\\
c & d
\end{pmatrix}(x)=\left\{\begin{matrix} \dfrac{ax+b}{cx+d} & if & x\in\mathbb{C}\\ \\ \dfrac{a}{c} & if & x=\infty.\end{matrix}\right.
\end{align*}  
The elements of the Mobius group are called {\it Mobius transformation}.
\begin{defin}
We say that $f\in\End(\calH)$ is a {\it standard operator} if $f$ is invertible and if for all admissible pair $(p,q)\in(\bbRz)^2$, there exists an admissible pair $(r,s)\in(\bbRz)^2$ such that:
\begin{enumerate}

\item either \begin{align}\label{standard.op.cas1}
	f(\calH_{p,q})=\calH_{r,s} \ \ \text{and} \ \ f(\bar{\calH}_{p,q})=\bar{\calH}_{r,s},
	\end{align}
	or 
	\begin{align}\label{standard.op.cas2}
	f(\calH_{p,q})=\bar{\calH}_{r,s}, \ \ \text{and} \ \ f(\bar{\calH}_{p,q})=\calH_{r,s},
	\end{align}

\item there exists a Mobius transformation $M\in\PGL(2,\bbC)$ such that 
    \begin{align}
	M\pa{\tfrac{q}{p+q}}=\tfrac{s}{r+s}
	\end{align}
\end{enumerate}
\end{defin}

The scaling invariance property of the multiplicity spaces extends to the standard operators in the following sense : if $f\in\End(\calH)$ is a standard operator, then for all $\alp,\bet\in\bbZ_N^2$, there exists functions 
\begin{align*}
f_{\alp,\bet},\bar{f}_{\alp,\bet}  : \bbRzu\to \bbC
\end{align*}
such that 
\begin{align}\begin{split}
fh_{p,q}(u_\alp)=\sum_{\bet\in\bbZ_N^2}f_{\alp,\bet}\pa{\tfrac{q}{p+q}}h_{r,s}(\bar{u}_\bet),
\\
f\bar{h}_{p,q}(\bar{u}_\alp)=\sum_{\bet\in\bbZ_N^2}\bar{f}_{\alp,\bet}\pa{\tfrac{q}{p+q}}h_{r,s}(u_\bet),
\end{split}
\end{align}
if $f$ satisfies \eqref{standard.op.cas1}
and 
\begin{align}\begin{split}
fh_{p,q}(u_\alp)=\sum_{\bet\in\bbZ_N^2}f_{\alp,\bet}\pa{\tfrac{q}{p+q}}\bar{h}_{r,s}(\bar{u}_\bet),
\\
f\bar{h}_{p,q}(\bar{u}_\alp)=\sum_{\bet\in\bbZ_N^2}\bar{f}_{\alp,\bet}\pa{\tfrac{q}{p+q}}h_{r,s}(u_\bet),
\end{split}
\end{align}
if $f$ satisfies \eqref{standard.op.cas2}.

Following \cite{geer2012tetrahedral} we also consider the following two definitions.
\begin{defin}
An operator $f\in\End(\calH)$ is {\it grading-preserving} if for all admissible pairs $(p,q)\in(\bbRz)^2$ we have 
\begin{align*}
f(\calH_{p,q})\subset\calH_{p,q} \ \ \text{and} \ \ f(\bar{\calH}_{p,q})\subset\bar{\calH}_{p,q}.
\end{align*}
\end{defin}
\noindent
Clearly, the invertible grading-preserving operators are standard.

Let $\pi_{p,q}: \calH\to \calH_{p,q}$ and  $\bar{\pi}_{p,q}:\calH\to\bar{\calH}_{p,q}$
%
%
be the obvious projections. 
We provide $\calH$ with a symmetric bilinear pairing $\left<,\right>:\calH\ot \calH\to\bbC$ by
\begin{align*}
\left< u,v \right>= \sum_{\substack{(p,q)\in(\bbRz)^2 \\ \text{admissible}}}(\left< \bar{\pi}_{p,q}(u) \pi_{p,q}(v)\right>+\left< \bar{\pi}_{p,q}(v) \pi_{p,q}(u)\right>)
\end{align*}
for any $u,v\in \calH$.
\begin{defin}
A {\it transpose} of $f\in\End(\calH)$ is a map $f^*\in\End(\calH)$ such that $\left<fu,v\right>=\left<u,f^*v\right>$ for all $u,v\in \calH$. We say that $f\in\End(\calH)$ is {\it symmetric} if $f^*=f$.
\end{defin}
Since $\calH_{p,q}$ and $\bar{\calH}_{p,q}$ are dual vector spaces, the transpose $f^*$ of $f\in\End(\calH)$, if it exists, is unique and $(f^*)^*=f$.

If $f\in\End(\calH)$ is standard, the equalities \eqref{standard.op.cas1} and \eqref{standard.op.cas2} ensure that $f^*$ exists. Moreover, in that case, $f^*$ is also standard. 

\subsection{The operators \texorpdfstring{$A$ and $B$}{A and B} and their transpose}
Following \cite{geer2012tetrahedral}, we define the operators $A,B\in\End(\calH)$ by 
\begin{align}\label{def.A}
Au=\sum_{\substack{(p,q)\in(\bbRz)^2 \\ \text{admissible}}} (\Id_\calV\ot\bar{\pi}_{p,q}(u))(b_{-p}\ot\Id_\calV)+(d_{-p}\ot\Id_\calV)(\Id_\calV\ot\pi_{p,q}(u)),
\end{align}
\begin{align}\label{def.B}
Bu=\sum_{\substack{(p,q)\in(\bbRz)^2 \\ \text{admissible}}} (\bar{\pi}_{p,q}(u)\ot\Id_\calV)(\Id_\calV\ot b_q)+(\Id_\calV\ot d_q)(\pi_{p,q}(u)\ot\Id_\calV).
\end{align}
For each $u\in \calH$, there are only finitely many non-zero terms in these sums, since $u$ has only finitely many non-zero components $\pi_{p,q}(u)$ and $\bar{\pi}_{p,q}(u)$.

Using \eqref{del.bp.dp}, one can easily prove that the operators $A$ and $B$ are involutive (see \cite[Lemma 3]{geer2012tetrahedral}).
Hence, from their definition, we clearly have the following equalities
\begin{align}\begin{split}\label{AB.incl}
A(\calH_{p,q})=\bar{\calH}_{-p,p+q}, & \ \ A(\bar{\calH}_{p,q})=\calH_{-p,p+q}, \\
 B(\calH_{p,q})=\bar{\calH}_{p+q,-q}, & \ \  B(\bar{\calH}_{p,q})=\calH_{p+q,-q}.
\end{split}\end{align}
Moreover, $A$ and $B$ are both standard operators. Indeed, we have 
\begin{align*}
\begin{pmatrix}
0 & 1
\\
1 & 0
\end{pmatrix}
\pa{\tfrac{q}{p+q}}
=
\tfrac{p+q}{q}
\ \ \text{and} \ \
\begin{pmatrix}
1 & 0
\\
1 & -1
\end{pmatrix}
\pa{\tfrac{q}{p+q}}
=
\tfrac{-q}{p}.
\end{align*}
The equalities \eqref{AB.incl} ensure us that we have the following for $A^*$ and $B^*$ 
\begin{align}\begin{split}\label{A*B*.incl}
A^*(\calH_{p,q})=\bar{\calH}_{-p,p+q}, & \ \ A^*(\bar{\calH}_{p,q})=\calH_{-p,p+q}, \\
 B^*(\calH_{p,q})=\bar{\calH}_{p+q,-q}, & \ \  B^*(\bar{\calH}_{p,q})= \calH_{p+q,-q}.
\end{split}\end{align}
%
The equalities \eqref{AB.incl} and \eqref{A*B*.incl} ensure us that for all $\alp,\bet\in\bbZ_N^2$ there exists functions
\begin{align*}
A_{\alp,\bet},\bar{A}_{\alp,\bet},A^*_{\alp,\bet},\bar{A}^*_{\alp,\bet}  : \bbRzu\to \bbC
\end{align*}
and 
\begin{align*}
B_{\alp,\bet}, \bar{B}_{\alp,\bet},B^*_{\alp,\bet},\bar{B}^*_{\alp,\bet}  : \bbRzu\to \bbC
\end{align*}
such that for all admissible pairs $(p,q)\in(\bbRz)^2$ we have
\begin{align*}\begin{split}
Ah_{p,q}(u_\alp)=\sum_{\bet\in\bbZ_N^2}A_{\alp,\bet}\pa{\tfrac{q}{p+q}}\bar{h}_{-p,p+q}(\bar{u}_\bet), 
\\
A\bar{h}_{p,q}(\bar{u}_\alp)=\sum_{\bet\in\bbZ_N^2}\bar{A}_{\alp,\bet}\pa{\tfrac{q}{p+q}}h_{-p,p+q}(u_\bet), 
\\
A^*h_{p,q}(u_\alp)=\sum_{\bet\in\bbZ_N^2}A^*_{\alp,\bet}\pa{\tfrac{q}{p+q}}\bar{h}_{-p,p+q}(\bar{u}_\bet), 
\\
A^*\bar{h}_{p,q}(\bar{u}_\alp)=\sum_{\bet\in\bbZ_N^2}\bar{A}^*_{\alp,\bet}\pa{\tfrac{q}{p+q}}h_{-p,p+q}(u_\bet).
\end{split}\end{align*}
and 
\begin{align*}\begin{split}
Bh_{p,q}(u_\alp)=\sum_{\bet\in\bbZ_N^2}B_{\alp,\bet}\pa{\tfrac{q}{p+q}}\bar{h}_{p+q,-q}(\bar{u}_\bet),
\\
B\bar{h}_{p,q}(\bar{u}_\alp)=\sum_{\bet\in\bbZ_N^2}\bar{B}_{\alp,\bet}\pa{\tfrac{q}{p+q}}h_{p+q,-q}(u_\bet),
\\
B^*h_{p,q}(u_\alp)=\sum_{\bet\in\bbZ_N^2}B^*_{\alp,\bet}\pa{\tfrac{q}{p+q}}\bar{h}_{p+q,-q}(\bar{u}_\bet), 
\\
B^*\bar{h}_{p,q}(\bar{u}_\alp)=\sum_{\bet\in\bbZ_N^2}\bar{B}^*_{\alp,\bet}\pa{\tfrac{q}{p+q}}h_{p+q,-q}(u_\bet).
\end{split}\end{align*}

We define 
\begin{align*}
\epsilon_N=\left\{\begin{matrix} 1 & if & N=1 \mod 4 \\ i & if & N=3 \mod 4\end{matrix}\right.
\end{align*}
and we use the following result.
\begin{proposition}[\cite{karemera2016quantum}, Proposition 3.4]\label{det.A.et.B}
There exist $\mathtt{a},\mathtt{b}\in\bbZ_N$ such that for all $x\in\bbRzu$
and all $\alp,\bet\in\bbZ_N^2$ we have 
\begin{align*}
A_{\alp,\bet}(x)= &\, \epsilon_N^2x^{\frac{2(N-1)}{N}}\frac{\del_{\alp_2,-\bet_2}}{N}\om^{-\frac{1}{2}\pa{t(\alp_1-\alp_2-\bet_1)(\alp_1-\alp_2-\bet_1-1)-\alp_1(\alp_1-3)-(t+1)\alp_2-\bet_1(\bet_1+2t+1)}+\mathtt{a}},
\\
\bar{A}_{\alp,\bet}(x)= &\, \epsilon_N^{-2}x^{\frac{2(N-1)}{N}}\del_{\alp_2,-\bet_2}\om^{\frac{1}{2}\left(t(\bet_1+\alp_2-\alp_1)(\bet_1+\alp_2-\alp_1-1)+\bet_1(\bet_1-3)-(t+1)\alp_2+\alp_1(\alp_1+2t+1)\right)-\mathtt{a}},
\\
A^*_{\alp,\bet}(x)= &\, \epsilon_N^2x^{-\frac{2(N-1)}{N}}\frac{\del_{\alp_2,-\bet_2}}{N}\om^{-\frac{1}{2}\left(t(\bet_1+\alp_2-\alp_1)(\bet_1+\alp_2-\alp_1-1)+\bet_1(\bet_1-3)-(t+1)\alp_2+\alp_1(\alp_1+2t+1)\right)+\mathtt{a}},
\\
\bar{A}^*_{\alp,\bet}(x)= &\, \epsilon_N^{-2}x^{-\frac{2(N-1)}{N}}\del_{\alp_2,-\bet_2}\om^{\frac{1}{2}\left(t(\alp_1-\alp_2-\bet_1)(\alp_1-\alp_2-\bet_1-1)-\alp_1(\alp_1-3)-(t+1)\alp_2-\bet_1(\bet_1+2t+1)\right)-\mathtt{a}},
\end{align*}
and
\begin{align*}
B_{\alp,\bet}(x)= &\; (1-x)^{\frac{2(N-1)}{N}}\del_{\alp,-\bet}\om^{-\frac{1}{2}\pa{(t+1)(\alp_1-\alp_2)(\alp_1-\alp_2-2)+\alp_1^2-\alp_2}+\mathtt{b}},
\\
\bar{B}_{\alp,\bet}(x)= &\; (1-x)^{\frac{2(N-1)}{N}}\del_{\alp,-\bet}\om^{\frac{1}{2}\pa{(t+1)(\alp_1-\alp_2)(\alp_1-\alp_2+2)+\alp_1^2+\alp_2}-\mathtt{b}},
\\
B^*_{\alp,\bet}(x)= &\; (1-x)^{-\frac{2(N-1)}{N}}\del_{\alp,-\bet}\om^{-\frac{1}{2}\pa{(t+1)(\alp_1-\alp_2)(\alp_1-\alp_2+2)+\alp_1^2+\alp_2}+\mathtt{b}},
\\
\bar{B}^*_{\alp,\bet}(x)= &\; (1-x)^{\frac{-2(N-1)}{N}}\del_{\alp,-\bet}\om^{\frac{1}{2}\pa{(t+1)(\alp_1-\alp_2)(\alp_1-\alp_2-2)+\alp_1^2-\alp_2}-\mathtt{b}}.
\end{align*}
\end{proposition}

\begin{rem}
In \cite[Proposition 3.4]{karemera2016quantum}, the equalities are given up to integer powers of $\om$. The only difference with the above statement is that these integer powers of $\om$ are made explicit (see \cite[Chapter 2]{karemera2016quantum} for details).
\end{rem}


\subsection{The operators \texorpdfstring{$L,R$ and $C$}{L, R and C}} 

Following \cite{geer2012tetrahedral}, the operators $L,R$ and $C$ are defined as
\begin{align}\label{def.LRC}
L=A^*A, \ \ R=B^*B,\ \ C=(AB)^3\in\End(\calH).
\end{align}
The operators $L,R$ and $C$ are clearly invertible and by the equalities \eqref{AB.incl} and \eqref{A*B*.incl}, we easily see that they are grading-preserving. Hence, these operators are also standard.

Moreover, these operators are symmetric. It is clear for $L$ and $R$. For $C$, we use \cite[Lemma 5]{geer2012tetrahedral} which states that 
\begin{align*}
(ABA)^*=BAB.
\end{align*}
Hence we have 
\begin{align*}
C=(AB)^3=ABABAB=ABA(ABA)^*.
\end{align*}

Now we can determine these operators using the functions
\begin{align*}
L_{\alp,\bet},\bar{L}_{\alp,\bet},R_{\alp,\bet},\bar{R}_{\alp,\bet},C_{\alp,\bet},\bar{C}_{\alp,\bet}  : \bbRzu\to \bbC
\end{align*}
satisfying 
\begin{align*}
Lh_{p,q}(u_\alp)=\sum_{\bet\in\bbZ_N^2}L_{\alp,\bet}\pa{\tfrac{q}{p+q}}h_{p,q}(u_\bet), 
\\
L\bar{h}_{p,q}(\bar{u}_\alp)=\sum_{\bet\in\bbZ_N^2}\bar{L}_{\alp,\bet}\pa{\tfrac{q}{p+q}}\bar{h}_{p,q}(\bar{u}_\bet), 
\\
Rh_{p,q}(u_\alp)=\sum_{\bet\in\bbZ_N^2}R_{\alp,\bet}\pa{\tfrac{q}{p+q}}h_{p,q}(u_\bet), 
\\
R\bar{h}_{p,q}(\bar{u}_\alp)=\sum_{\bet\in\bbZ_N^2}\bar{R}_{\alp,\bet}\pa{\tfrac{q}{p+q}}\bar{h}_{p,q}(\bar{u}_\bet),
\\
Ch_{p,q}(u_\alp)=\sum_{\bet\in\bbZ_N^2}C_{\alp,\bet}\pa{\tfrac{q}{p+q}}h_{p,q}(u_\bet), 
\\
C\bar{h}_{p,q}(\bar{u}_\alp)=\sum_{\bet\in\bbZ_N^2}\bar{C}_{\alp,\bet}\pa{\tfrac{q}{p+q}}\bar{h}_{p,q}(\bar{u}_\bet),
\end{align*}
for all admissible pairs $(p,q)\in(\bbRz)^2$.


\begin{proposition}
For all $x\in\bbRzu$ and all $\alp,\bet\in\bbZ_N^2$ we have
\begin{align*}
L_{\alp,\bet}(x)= &\; x^{\frac{4(N-1)}{N}}\del_{\alp_1+2t,\bet_1}\del_{\alp_2,\bet_2}\om^{2t\alp_1+\alp_2-4t-3},
\\
\bar{L}_{\alp,\bet}(x)= &\; x^{\frac{4(N-1)}{N}}\del_{\alp_1-2t,\bet_1}\del_{\alp_2,\bet_2}\om^{2t\alp_1+\alp_2+1},
\\
R_{\alp,\bet}(x)= &\; (1-x)^{\frac{4(N-1)}{N}}\del_{\alp,\bet}\om^{2(t+1)(\alp_1-\alp_2)+\alp_2},
\\
\bar{R}_{\alp,\bet}(x)= &\; (1-x)^{\frac{4(N-1)}{N}}\del_{\alp,\bet}\om^{2(t+1)(\alp_1-\alp_2)+\alp_2},
\\
C_{\alp,\bet}(x)= &\; \del_{\alp,\bet}\om^{-(2t+1)\alp_2+3(\mathtt{b}-\mathtt{a})},
\\
\bar{C}_{\alp,\bet}(x)= &\; \del_{\alp,\bet}\om^{-(2t+1)\alp_2+3(\mathtt{a}-\mathtt{b})},
\end{align*}
with $\mathtt{a},\mathtt{b}\in\bbZ_N$ given in Proposition \ref{det.A.et.B}.
\end{proposition}
\begin{proof}
Using Proposition \ref{det.A.et.B}, a straightforward computation leads to the results.
\end{proof}

\begin{rem} Following \cite[Remark 41]{geer2012tetrahedral}, an immediate consequence of the fact that the operator $C$ is non trivial is that the category of $\calA_{\om,t}$-modules is non-pivotal.%
\end{rem}

\subsection{The operators \texorpdfstring{$L^\frac{1}{2},R^\frac{1}{2}$}{L\unichar{"02C6}\{\unichar{"00BD}\}, R\unichar{"02C6}\{\unichar{"00BD}\}} and \texorpdfstring{$C^\frac{1}{2}$}{C\unichar{"02C6}\{\unichar{"00BD}\}}}

Now we can fix square roots of the operators $L,R$ and $C$. These are the key operators mentioned in the beginning of this Section that will allow us to extend our $\Psi$-system into a $\hat{\Psi}$-system. 

\begin{proposition}\label{sqrtLRC}
The square roots of the operators $L,R$ and $C$ can be chosen to be, respectively, the grading-preserving operators given, for all $x\in\bbRzu$ 
and all $\alp,\bet\in\bbZ_N^2$, by
\begin{align*}
L^{\frac{1}{2}}_{\alp,\bet}(x)= &\; x^{\frac{2(N-1)}{N}}\del_{\alp_1+t,\bet_1}\del_{\alp_2,\bet_2}\om^{\frac{1}{2}\pa{2t\alp_1+\alp_2-3t-2}},
\\
\bar{L}^{\frac{1}{2}}_{\alp,\bet}(x)= &\; x^{\frac{2(N-1)}{N}}\del_{\alp_1-t,\bet_1}\del_{\alp_2,\bet_2}\om^{\frac{1}{2}\pa{2t\alp_1+\alp_2-t}},
\\
R^{\frac{1}{2}}_{\alp,\bet}(x)= &\; (1-x)^{\frac{2(N-1)}{N}}\del_{\alp,\bet}\om^{(t+1)(\alp_1-\alp_2)+\frac{1}{2}\alp_2},
\\
\bar{R}^{\frac{1}{2}}_{\alp,\bet}(x)= &\; (1-x)^{\frac{2(N-1)}{N}}\del_{\alp,\bet}\om^{(t+1)(\alp_1-\alp_2)+\frac{1}{2}\alp_2},
\\
C^{\frac{1}{2}}_{\alp,\bet}(x)= &\; \del_{\alp,\bet}\om^{-(t+\frac{1}{2})\alp_2+\frac{3(\mathtt{b}-\mathtt{a})}{2}},
\\
\bar{C}^{\frac{1}{2}}_{\alp,\bet}(x)= &\; \del_{\alp,\bet}\om^{-\pa{t+\frac{1}{2}}\alp_2+\frac{3(\mathtt{a}-\mathtt{b})}{2}}.
\end{align*}
with $\mathtt{a},\mathtt{b}\in\bbZ_N$ given in Proposition \ref{det.A.et.B}.
\end{proposition}
\begin{proof}
Let us write for any $L=L_0L_1$ as a product of commuting operators $L_0$ and $L_1$ such that, for any admissible pair $(p,q)\in(\bbRz)^2$ we have 
\begin{align*}
L_0|_{\calH_{p,q}\oplus\bar{\calH}_{p,q}}= \pa{\tfrac{q}{p+q}}^{\frac{4(N-1)}{N}} \ \  \text{and} \ \ L_1^N|_\calH=\Id_\calH.
\end{align*}
Since $N$ is odd, we set $L^{\frac{1}{2}}=L_0^{\frac{1}{2}}L_1^{\frac{N+1}{2}}$. We do similarly for $R^{\frac{1}{2}}$ and $C^\frac{1}{2}$.
\end{proof}

\begin{rem}
In \cite[equation (35)]{geer2012tetrahedral} the square root of $L$ is defined by
\begin{align}\label{L.sqrt..a.la.GKT}
L^{\frac{1}{2}}=BAR^{-\frac{1}{2}}AB 
\end{align}
and it is shown that it implies that $(L^{\frac{1}{2}})^2=L$. Although $L^{\frac{1}{2}}$ has not been defined this way in our case, a straightforward computation shows that equality \eqref{L.sqrt..a.la.GKT} holds true.
\end{rem}

\section{The \texorpdfstring{$6j$}{6\unichar{"006A}}-symbols}\label{sec.6j.symb}

Besides the operators $L,R$ and $C$, an other key component of a $\Psi$-system is its associated $6j$-symbols. In what follows we are going to define them and determine them. As we have already said, the operator $S(x)$ will have an essential role.

\subsection{Definition of the \texorpdfstring{$6j$}{6j}-symbols}
Following \cite{kashaev1999matrix} we consider, for any $x\in\bbRzu$, the algebra morphism 
%
   $ \Del_x: \mathcal{A}\to\mathcal{A}\ot\mathcal{A}$
%
by 
\begin{align}\label{def.Del.x}
\Del_x(a) =  S(x)(a\ot\Id_\calV)S(x)^{-1}
\end{align}
and the following function
\begin{equation*}
\begin{array}{lccc}
*: & (\bbRzu)^2 & \longrightarrow & \bbRzu \\
    & (x,y) & \longmapsto & \dfrac{y-xy}{1-xy} 
\end{array}
\end{equation*}
The latter function is well defined only for pairs $(x,y)\in(\bbRzu)^2$ such that $x\neq y^{-1}$.
\begin{defin}
We say that a pair $(x,y)\in(\bbRzu)^2$ is {\it compatible} if $x\neq y^{-1}$.
\end{defin}
Using the terminology of \cite{kashaev1999matrix}, we show in the next result that $\calA$ is endowed with a {\it generalized comultiplication}. 
\begin{lemma}\label{coass.gen}
For all compatible pairs $(x,y)\in(\bbRzu)^2$, we have
 \begin{align}\label{formule.des.Del}
 \pa{\Del_{x*y}\ot\Id_\calV}\Del_{xy}=\pa{\Id_\calV\ot\Del_{x}}\Del_{y}
 \end{align}
\end{lemma}
\begin{proof}
If $(x,y)\in(\bbRzu)^2$ is a compatible pair, then $\Del_{xy}$ and $\Del_{x*y}$ are well defined functions. For any $x\in\bbRzu$, Proposition \ref{prop.base.Hpq} implies the following equalities

\begin{equation*}\begin{array}{l}
\Del_x(X)=X_1X_2, 
\ \
\Del_x(Y)=(1-x)^{\frac{1}{N}}Y_1+x^\frac{1}{N}X_1Y_2 
\\
\Del_x(U)=(1-x)^{\frac{2}{N}}U_1+x^{\frac{2}{N}}X_1^{t+1}U_2+(x-x^2)^\frac{1}{N}X_1^tY_1Z_2,
\\
\Del_x(V)=(1-x)^{\frac{2}{N}}V_1+x^{\frac{2}{N}}X_1^{t+1}V_2+(x-x^2)^\frac{1}{N}Z_1X_1Y_2,
\end{array}\end{equation*}
where $Z=(U+V)Y^{-1}$. Using these equalities, a straightforward computation shows that equation \eqref{formule.des.Del} holds true for $a\in\brkt{X,Y,U,V}$. This ends the proof since $\brkt{X,Y,U,V}$ is a generating set of $\calA$.
\end{proof}

Let $(x,y)\in(\bbRzu)^2$ be compatible. Using \eqref{def.Del.x}, we have for all $a\in\calA$
\begin{align*}
\pa{\Del_{x*y}\ot\Id_\calV}\Del_{xy}(a) & =S_{12}(x*y)S_{13}(xy)(a\ot\Id_\calV\ot\Id_\calV)\pa{S_{12}(x*y)S_{13}(xy)}^{-1}
\end{align*}
and
\begin{align*}
\pa{\Id_\calV\ot\Del_{x}}\Del_{y}(a) & =S_{23}(x)S_{12}(y)(a\ot\Id_\calV\ot\Id_\calV)\pa{S_{23}(x)S_{12}(y)}^{-1}.
\end{align*}
Therefore, by Lemma \ref{coass.gen}, the following equality holds in $\calA^{\ot 3}$ for all $a\in\calA$
%
%
\begin{align*}
& \pa{S_{23}(x)S_{12}(y)}^{-1}S_{12}(x*y)S_{13}(xy)(a\ot\Id_\calV\ot\Id_\calV)
\\
& \quad\,=(a\ot\Id_\calV\ot\Id_\calV)\pa{S_{23}(x)S_{12}(y)}^{-1}S_{12}(x*y)S_{13}(xy).
\end{align*}
Since the center of $\calA$ is trivial, the former equality implies the existence of an element $T(x,y)\in\calA^{\ot 2}$ such that 
\begin{align*}
\pa{S_{23}(x)S_{12}(y)}^{-1}S_{12}(x*y)S_{13}(xy)T_{23}(x,y)=\Id_\calV\ot\Id_\calV\ot\Id_\calV.
\end{align*}
Hence  we have 
\begin{align}\label{pent.T}
S_{23}(x)S_{12}(y)=S_{12}(x*y)S_{13}(xy)T_{23}(x,y).
\end{align}
\begin{defin}
The operator $T(x,y)\in\calA^{\ot 2}$ and its inverse are called {\it $6j$-symbols}.
\end{defin}
\begin{remark}\label{rmk.T.as.6j}
The operators $T(x,y)$ and $T(x,y)^{-1}$ correspond to the $6j$-symbols (positive and negative respectively) defined in \cite{geer2012tetrahedral} in the following way : for $p,q,r\in\bbRz$ such that $x=\tfrac{r}{q+r}$ and $y=\tfrac{q+r}{p+q+r}$, we have 
\begin{align*}
(\bar{v}\ot\bar{u})T(x,y)(v\ot u)=\begin{Bmatrix} p & q & p+q \\ r & p+q+r & q+r \end{Bmatrix}(\bar{u}\ot\bar{v}\ot u\ot v)\in\bbC
\end{align*} 
where 
$$\bar{u}\ot\bar{v}\ot u\ot v\in \bar{\calH}_{p+q,r}\ot\bar{\calH}_{p,q}\ot \calH_{q,r}\ot \calH_{p,q+r}$$
and
\begin{align*}
(\bar{u}'\ot\bar{v}')T(x,y)^{-1}(u'\ot v')=\begin{Bmatrix} p & q & p+q \\ r & p+q+r & q+r \end{Bmatrix}^{-}(\bar{u}'\ot\bar{v}'\ot u'\ot v')\in\bbC
\end{align*}
where 
$$\bar{u}'\ot\bar{v}'\ot u'\ot v'\in \bar{\calH}_{p,q+r}\ot\bar{\calH}_{q,r}\ot \calH_{p,q}\ot \calH_{p+q,r}.$$ 
Thus, $T(x,y)$ and $T(x,y)^{-1}$ are interpreted as elements of $\End(\calH^{\ot2})$, or more precisely, 
\begin{align*}
T(x,y):  \calH_{p,q+r}\ot\calH_{q,r}  \longrightarrow  \calH_{p,q}\ot\calH_{p+q,r},
\\
T(x,y)^{-1}:  \calH_{p,q}\ot\calH_{p+q,r}  \longrightarrow  \calH_{p,q+r}\ot\calH_{q,r}.
\end{align*}
\end{remark}
%

\subsection{Determination of the \texorpdfstring{$6j$}{6j}-symbols}

The following Theorem, proved by Faddeev and Kashaev in \cite{faddeev1994quantum}, is central to determine the $6j$-symbols $T(x,y)$. For the sake of simplicity, our statement is adapted to our context. It is therefore slightly different from the original one. 
\begin{theorem}[\cite{faddeev1994quantum}]\label{TQD}
Let $(x,y)\in(\bbRzu)^2$ be compatible and $\sfU,\sfV$ be two operators such that $\sfU\sfV=\om^{-1}\sfV\sfU$ and $\sfU^N=\sfV^N=-1$, then
\begin{align*}
\Psi_x(\sfU)\Psi_y(\sfV)=\Psi_{x*y}(\sfV)\Psi_{xy}(-\sfV\sfU)\Psi_{y*x}(\sfU).
\end{align*}
\end{theorem} 

\begin{lemma}\label{L2}
Let $\sfU,\sfV\in\calA$ be such that $\sfU^N=\sfV^N=1$ and $\sfU\sfV=\om \sfV\sfU$.Then 
\begin{align*}
 L_{23}(\sfU,\sfV)L_{12}(\sfU,\sfV)
 =L_{12}(\sfU,\sfV)L_{13}(\sfU,\sfV)L_{23}(\sfU,\sfV).
\end{align*} 
\end{lemma}

\begin{proof}
For all $i\in\bbZ_N$ we have 
%
%
\begin{align*}
& \sum_{k\in\bbZ_N}\pa{\frac{1}{N}\sum_{j\in\bbZ_N}\om^{-ij}\sfV^j}\pa{\frac{1}{N}\sum_{l\in\bbZ_N}\om^{-kl}\sfV^l} 
 = \frac{1}{N^2}\sum_{l,k\in\bbZ_N}\om^{-kl}\pa{\sum_{j\in\bbZ_N}\om^{-ij}\sfV^{j+l}} 
\\
& = \frac{1}{N^2}\sum_{l,k\in\bbZ_N}\om^{l(i-k)}\pa{\sum_{j\in\bbZ_N}\om^{-ij}\sfV^{j}}
= \pa{\sum_{j\in\bbZ_N}\om^{-ij}\sfV^{j}}\frac{1}{N^2}\sum_{l,k\in\bbZ_N}\om^{l(i-k)} 
\\
& = \pa{\sum_{j\in\bbZ_N}\om^{-ij}\sfV^{j}}\frac{1}{N^2}\sum_{k\in\bbZ_N}N\del_{i,k} 
= \frac{1}{N}\sum_{j\in\bbZ_N}\om^{-ij}\sfV^j.
\end{align*}
Using the identity above and Lemma \ref{prop.L1}, we easily make the following computation
%
\begin{align*}
    & L_{23}(\sfU,\sfV)L_{12}(\sfU,\sfV)  =  \frac{1}{N}\sum_{i,j\in\bbZ_N}\om^{-ij}\sfU_2^i\sfV_3^jL_{12}(\sfU,\sfV) 
 \\
 & = \frac{1}{N}\sum_{i,j\in\bbZ_N}\om^{-ij}L_{12}(\sfU,\om^i\sfV)\sfU_2^i\sfV_3^j 
  = L_{12}(\sfU,\sfV)\frac{1}{N}\sum_{i,j\in\bbZ_N}\om^{-ij}\sfU_1^i\sfU_2^i\sfV_3^j 
  \\ 
&  = L_{12}(\sfU,\sfV)\sum_{i\in\bbZ_N}\sfU_1^i\sfU_2^i\sum_{k\in\bbZ_N}\pa{\frac{1}{N}\sum_{j\in\bbZ_N}\om^{-ij}\sfV_3^j}\pa{\frac{1}{N}\sum_{l\in\bbZ_N}\om^{-kl}\sfV_3^l} 
  \\
 & = L_{12}(\sfU,\sfV)\sum_{i,k\in\bbZ_N}\sfU_1^i\sfU_2^k\pa{\frac{1}{N}\sum_{j\in\bbZ_N}\om^{-ij}\sfV_3^j}\pa{\frac{1}{N}\sum_{l\in\bbZ_N}\om^{-kl}\sfV_3^l} 
  \\
&  = L_{12}(\sfU,\sfV)\pa{\frac{1}{N}\sum_{i,j\in\bbZ_N}\om^{-ij}\sfU_1^i\sfV_3^j}\pa{\frac{1}{N}\sum_{k,l\in\bbZ_N}\om^{-kl}\sfU_2^k\sfV_3^l} 
  \\
&  = L_{12}(\sfU,\sfV)L_{13}(\sfU,\sfV)L_{23}(\sfU,\sfV).
\end{align*}
\end{proof}

In the following Theorem, we determine the $6j$-symbols $T(x,y)$ by show that $T(x,y)=S(y*x)$ is a solution of the equation \eqref{pent.T}. This result was suggested without a proof in \cite{kashaev1999matrix}.  

\begin{theorem}
$S_{23}(x)S_{12}(y)= S_{12}(x*y)S_{13}(xy)S_{23}(y*x)$.
\end{theorem}

\begin{proof} 
The following commutation relations hold true 
\begin{figure}[htb]\label{commut.E12's.E23's}
\begin{center}
\includegraphics[width=0.75\linewidth]{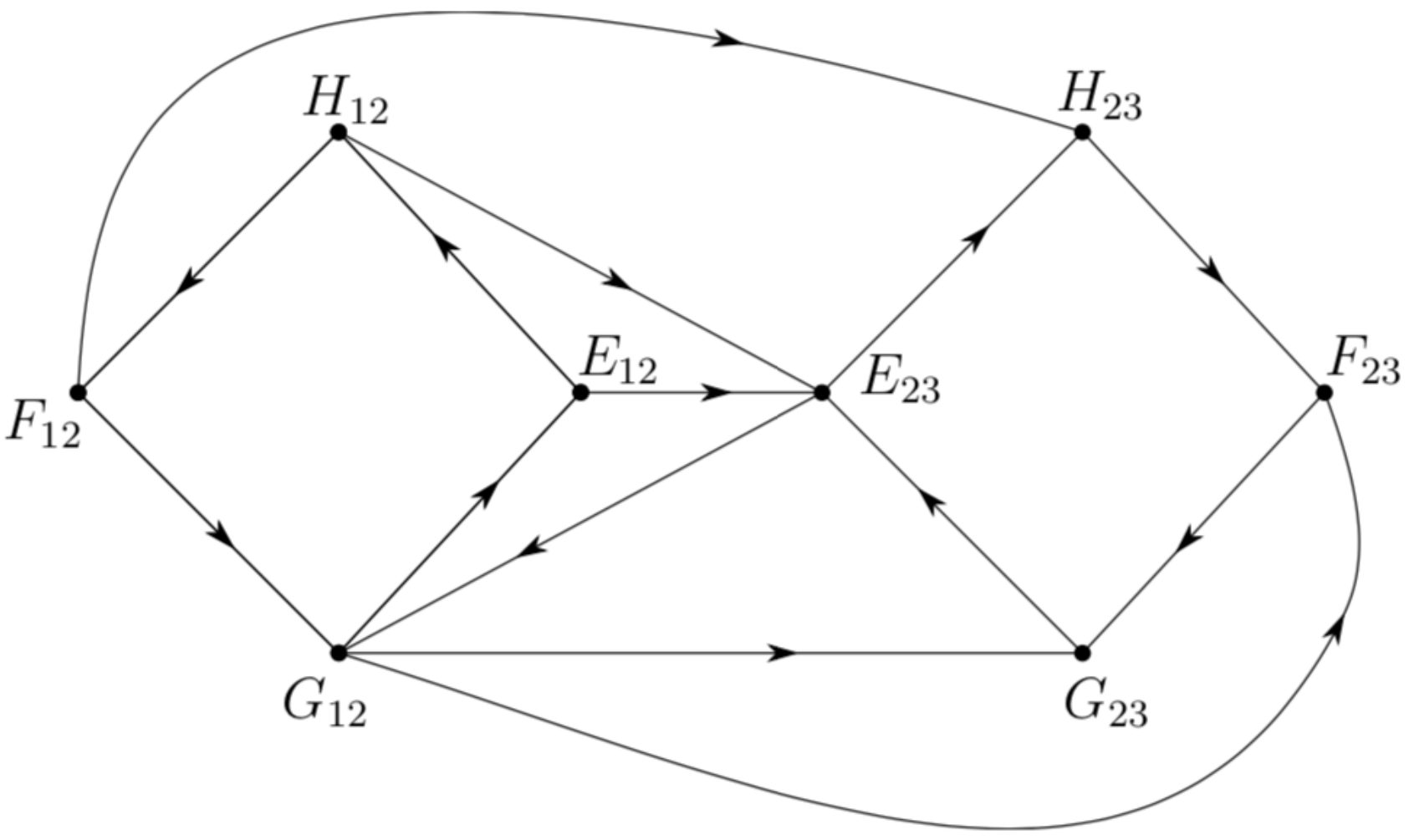}
\end{center}
\end{figure}

\noindent
Using Lemma \ref{prop.L1} and Lemma \ref{L2} for $\sfU=U^tV$, we see that the proposition is equivalent to the following equality
%
%

\begin{align}\begin{split}\label{equ1}
&  \Psi_x(E_{23})\Psi_x(F_{23})\Psi_x(G_{23})\Psi_x(H_{23})\Psi_y(E_{12})\Psi_y(F_{12})\Psi_y(G_{12})\Psi_y(H_{12}) 
\\
& =   \Psi_{x*y}(E_{12})\Psi_{x*y}(F_{12})\Psi_{x*y}(G_{12})\Psi_{x*y}(H_{12})\Psi_{xy}(X_2E_{13}) 
 \Psi_{xy}(X_2^tF_{13})
\\  
& \quad\, \times  \Psi_{xy}(X_2^tG_{13})\Psi_{xy}(X_2H_{13})\Psi_{y*x}(E_{23})\Psi_{y*x}(F_{23})\Psi_{y*x}(G_{23})\Psi_{y*x}(H_{23}).
\end{split}\end{align}
%
Since 
\begin{align*}
X_2E_{13} = & -E_{12}E_{23}, & X_2^tF_{13} = & -G_{12}F_{23}, \\
X_2^tG_{13} = & -G_{12}G_{23}, & X_2H_{13} = & -H_{12}E_{23},
\end{align*}
the former equality is equivalent to
 \begin{align*}
& \Psi_x(E_{23})\Psi_x(F_{23})\Psi_x(G_{23})\Psi_x(H_{23})\Psi_y(E_{12})\Psi_y(F_{12})\Psi_y(G_{12})\Psi_y(H_{12})   
 \\
& =   \Psi_{x*y}(E_{12})\Psi_{x*y}(F_{12})\Psi_{x*y}(G_{12})\Psi_{x*y}(H_{12})\Psi_{xy}(-E_{12}E_{23})\Psi_{xy}(-G_{12}F_{23})
 \\ 
& \quad\, \times  \Psi_{xy}(-G_{12}G_{23})\Psi_{xy}(-H_{12}E_{23})\Psi_{y*x}(E_{23})\Psi_{y*x}(F_{23})\Psi_{y*x}(G_{23})\Psi_{y*x}(H_{23}).  
 \end{align*}
%
On the left hand side, we can apply Theorem \ref{TQD} to $E_{12}$ and $E_{23}$. Indeed, $E_{12}$ commutes with $F_{23}$, $G_{23}$ and $H_{23}$ and $E_{23}E_{12}=\om^{-1}E_{12}E_{23}$. This gives
\begin{align*}
& \Psi_x(E_{23})\Psi_x(F_{23})\Psi_x(G_{23})\Psi_x(H_{23})\Psi_y(E_{12})\Psi_y(F_{12})\Psi_y(G_{12})\Psi_y(H_{12}) 
\\
&=   \Psi_{x*y}(E_{12})\Psi_{xy}(-E_{12}E_{23})\Psi_{y*x}(E_{23})\Psi_x(F_{23})\Psi_x(G_{23})\Psi_x(H_{23}) \Psi_y(F_{12}) \Psi_y(G_{12})
\\
& \quad\, \times \Psi_y(H_{12}). 
\end{align*}
%
On the right hand side, since $E_{12}E_{23}$ commutes with $F_{12}$, $G_{12}$ and $H_{12}$, we have
\begin{align*}
& \Psi_{x*y}(E_{12})\Psi_{x*y}(F_{12})\Psi_{x*y}(G_{12})\Psi_{x*y}(H_{12})\Psi_{xy}(-E_{12}E_{23})
 \Psi_{xy}(-G_{12}F_{23}) 
 \\ 
& \quad\,\times  \Psi_{xy}(-G_{12}G_{23})\Psi_{xy}(-H_{12}E_{23}) \Psi_{y*x}(E_{23})\Psi_{y*x}(F_{23})\Psi_{y*x}(G_{23})\Psi_{y*x}(H_{23}) 
 \\
&=  \Psi_{x*y}(E_{12})\Psi_{xy}(-E_{12}E_{23})\Psi_{x*y}(F_{12})\Psi_{x*y}(G_{12})\Psi_{x*y}(H_{12})
 \Psi_{xy}(-G_{12}F_{23})
 \\ 
 & \quad\,\times \Psi_{xy}(-G_{12}G_{23})\Psi_{xy}(-H_{12}E_{23}) \Psi_{y*x}(E_{23})\Psi_{y*x}(F_{23})\Psi_{y*x}(G_{23})\Psi_{y*x}(H_{23})
 \end{align*}
 %
Hence, (\ref{equ1}) is equivalent to 
%
\begin{equation}
\begin{aligned}\label{equ2}
&  \Psi_{y*x}(E_{23})\Psi_x(F_{23})\Psi_x(G_{23})\Psi_x(H_{23})\Psi_y(F_{12})\Psi_y(G_{12})\Psi_y(H_{12}) 
 \\
&=    \Psi_{x*y}(F_{12})\Psi_{x*y}(G_{12})\Psi_{x*y}(H_{12})\Psi_{xy}(-G_{12}F_{23})\Psi_{xy}(-G_{12}G_{23})  
 \\%
& \quad\,\times    \Psi_{xy}(-H_{12}E_{23}) \Psi_{y*x}(E_{23})\Psi_{y*x}(F_{23})\Psi_{y*x}(G_{23})\Psi_{y*x}(H_{23}) 
\end{aligned}
\end{equation}
%
On the left hand side of (\ref{equ2}), we have
\begin{align*}
 & \Psi_{y*x}(E_{23})\Psi_x(F_{23})\Psi_x(G_{23})\udl{\Psi_x(H_{23})\Psi_y(F_{12})}\Psi_y(G_{12})\Psi_y(H_{12})  
 \\
& = \Psi_{y*x}(E_{23})\Psi_x(F_{23})\Psi_x(G_{23})\Psi_{x*y}(F_{12})\udl{\Psi_{xy}(-F_{12}H_{23})}\Psi_{y*x}(H_{23})\Psi_y(G_{12})\Psi_y(H_{12})  \\
 & = \Psi_{y*x}(E_{23})\Psi_x(F_{23})\Psi_x(G_{23})\udl{\Psi_{x*y}(F_{12})}\Psi_{xy}(-E_{23}G_{12})\udl{\Psi_{y*x}(H_{23})}\Psi_y(G_{12})\Psi_y(H_{12})  
  \\
 & = \Psi_{x*y}(F_{12})\Psi_{y*x}(E_{23})\Psi_x(F_{23})\udl{\Psi_x(G_{23})\Psi_{xy}(-E_{23}G_{12})\Psi_y(G_{12})}\Psi_y(H_{12})\Psi_{y*x}(H_{23})  
  \\
 & = \Psi_{x*y}(F_{12})\udl{\Psi_{y*x}(E_{23})\Psi_x(F_{23})\Psi_{xy}(-E_{23}G_{12})\Psi_{x*y}(G_{12})}\Psi_{xy}(-G_{12}G_{23})\Psi_{y*x}(G_{23})
  \\
  & \quad\,\times \Psi_y(H_{12})\Psi_{y*x}(H_{23})  
  \\
 & = \Psi_{x*y}(F_{12})\Psi_x(F_{23})\Psi_y(G_{12})\Psi_x(E_{23})\Psi_{xy}(-G_{12}G_{23})\udl{\Psi_{y*x}(G_{23})\Psi_y(H_{12})}\Psi_{y*x}(H_{23})  
  \\
 & = \Psi_{x*y}(F_{12})\Psi_x(F_{23})\Psi_y(G_{12})\Psi_x(E_{23})\Psi_{xy}(-G_{12}G_{23})\Psi_y(H_{12})\Psi_{y*x}(G_{23})\Psi_{y*x}(H_{23}) 
\end{align*}
%
where we successively
\begin{enumerate}

	\item applied theorem \ref{TQD} to $H_{23}$ and $F_{12}$,
	
	\item used the equality $F_{12}H_{23}=E_{23}G_{12}$,
	
	\item used the fact that $F_{12}$ commutes with $E_{23}$, $F_{23}$ and $G_{23}$, and the fact that $H_{23}$ commutes with $G_{12}$ and $H_{12}$,
	
	\item used the fact that $G_{23}$ commutes with $E_{23}G_{12}$, and applied Theorem \ref{TQD} to $G_{23}$ and $G_{12}$,
	
	\item used the fact that $E_{23}$ commutes with $F_{23}$, and applied Theorem \ref{TQD} to $E_{23}$ and $G_{12}$,
	
	\item used the fact that $G_{23}$ commutes with $H_{12}$.

\end{enumerate}
On the right hand side of (\ref{equ2}), we have
\begin{align*}
&  \Psi_{x*y}(F_{12})\Psi_{x*y}(G_{12})\udl{\Psi_{x*y}(H_{12})\Psi_{xy}(-G_{12}F_{23})\Psi_{xy}(-G_{12}G_{23})\Psi_{xy}(-H_{12}E_{23}) \Psi_{y*x}(E_{23})} 
\\
& \quad\,\times\Psi_{y*x}(F_{23})\Psi_{y*x}(G_{23})\Psi_{y*x}(H_{23}) 
\\
& =  \Psi_{x*y}(F_{12})\udl{\Psi_{x*y}(G_{12})\Psi_{xy}(-G_{12}F_{23})\Psi_{xy}(-G_{12}G_{23})\Psi_x(E_{23})\Psi_y(H_{12})\Psi_{y*x}(F_{23})}  
\\
  &  \quad\,\times\Psi_{y*x}(G_{23})\Psi_{y*x}(H_{23}) 
\\
&  =  \Psi_{x*y}(F_{12})\Psi_x(F_{23})\Psi_y(G_{12})\udl{\Psi_{xy}(-G_{12}G_{23})\Psi_x(E_{23})}\Psi_y(H_{12})\Psi_{y*x}(G_{23})\Psi_{y*x}(H_{23}) 
\\
& =  \Psi_{x*y}(F_{12})\Psi_x(F_{23})\Psi_y(G_{12})\Psi_x(E_{23})\Psi_{xy}(-G_{12}G_{23})\Psi_y(H_{12})\Psi_{y*x}(G_{23})\Psi_{y*x}(H_{23}). 
\end{align*}
%
where we successively
\begin{enumerate}

	\item used the fact that $H_{12}$ commutes with $G_{12}F_{23}$ and $G_{12}G_{23}$, and applied Theorem \ref{TQD} to $H_{12}$ and $E_{23}$,
	
	\item used the fact that $F_{23}$ commutes with $H_{12}$ $E_{23}$ and $G_{12}G_{23}$, and applied Theorem \ref{TQD} to $G_{12}$ and $F_{23}$,
	
	\item used the fact that $E_{23}$ commutes with $G_{12}G_{23}$.
	
\end{enumerate}
\end{proof}

\begin{remark}\label{rmk.S.as.6j}
Setting $z=y*x$ and using Remark \ref{rmk.T.as.6j}, we have $z=\tfrac{pr}{(p+q)(q+r)}$, 
\begin{align*}
(\bar{v}\ot\bar{u})S(z)(v\ot u)=\begin{Bmatrix} p & q & p+q \\ r & p+q+r & q+r \end{Bmatrix}(\bar{u}\ot\bar{v}\ot u\ot v)\in\bbC
\end{align*}
where 
$$\bar{u}\ot\bar{v}\ot u\ot v\in \bar{\calH}_{p+q,r}\ot\bar{\calH}_{p,q}\ot \calH_{q,r}\ot \calH_{p,q+r}$$
and
\begin{align*}
(\bar{u}'\ot\bar{v}')S(z)^{-1}(u'\ot v')=\begin{Bmatrix} p & q & p+q \\ r & p+q+r & q+r \end{Bmatrix}^{-}(\bar{u}'\ot\bar{v}'\ot u'\ot v')\in\bbC
\end{align*}
where 
$$\bar{u}'\ot\bar{v}'\ot u'\ot v'\in \bar{\calH}_{p,q+r}\ot\bar{\calH}_{q,r}\ot \calH_{p,q}\ot \calH_{p+q,r}.$$  Thus, $S(z)$ and $S(z)^{-1}$ are interpreted as operators with the following source and target spaces  
\begin{align*}
S(z):  \calH_{p,q+r}\ot\calH_{q,r}  \longrightarrow  \calH_{p,q}\ot\calH_{p+q,r},
\\
S(z)^{-1}:   \calH_{p,q}\ot\calH_{p+q,r}  \longrightarrow  \calH_{p,q+r}\ot\calH_{q,r}.
\end{align*}
\end{remark}




\section{A \texorpdfstring{$\hat{\Psi}$}{\unichar{"03A8}\unichar{"0302}}-system in the category of \texorpdfstring{$\calA_{\om,t}$}{A\_\{\unichar{"03C9},t\}}-modules}\label{sec.hat.psi.sys.in.Awt}

In this Section, we show that a 3-manifold invariant can be defined. To do so, we are going to prove that \cite[Conjecture 42]{geer2012tetrahedral} holds true for the $\Psi$-system of Theorem \ref{thm.Psi.syst}. More precisely, we show that 
\begin{enumerate}
    \item the $\Psi$-system of Theorem \ref{thm.Psi.syst} extends to a $\hat{\Psi}$-system,
    \item a particular operator $\frakq\in\End(\mathcal{H})$ is equal to $q\Id_{\check{\mathcal{H}}}\oplus\, q^{-1}\Id_{\hat{\mathcal{H}}}$, where $q\in\bbC$. 
\end{enumerate}
%

\subsection{The \texorpdfstring{$\hat{\Psi}$}{\unichar{"03A8}\unichar{"0302}}-system}

Following \cite{geer2012tetrahedral}, the $\Psi$-system of Theorem \ref{thm.Psi.syst} extends to a $\hat{\Psi}$-system if two sets of equalities hold true. The first set involves the operators $A,B,L^{\frac{1}{2}},R^{\frac{1}{2}}$ and $C^{\frac{1}{2}}$: 
\begin{align}\begin{split}\label{psi.syst.eq1}
     \left(C^{\frac{1}{2}}\right)^{2}=C, & \quad A C^{\frac{1}{2}} A=B C^{\frac{1}{2}} B=C^{-\frac{1}{2}}, \quad L^{\frac{1}{2}}=BAR^{-\frac{1}{2}}AB, \\
\left(R^{\frac{1}{2}}\right)^{2}=R, & \quad B R^{\frac{1}{2}} B=R^{-\frac{1}{2}}, \quad R^{\frac{1}{2}} C^{\frac{1}{2}}=C^{\frac{1}{2}} R^{\frac{1}{2}}.
\end{split}
\end{align}
The second set of equalities involves the $6j$-symbols $S(z)$ as described in Remark~\ref{rmk.S.as.6j}:
\begin{align}\begin{split}\label{psi.syst.eq2}
C_1^{\frac{1}{2}}C_2^{\frac{1}{2}}S(z)&=S(z)C_1^{\frac{1}{2}}C_2^{\frac{1}{2}}, 
\\
L_1^{\frac{1}{2}}R_2^{\frac{1}{2}}S(z)&=S(z)L_1^{\frac{1}{2}}R_2^{\frac{1}{2}},
\\
R_1^{\frac{1}{2}}R_2^{\frac{1}{2}}S(z)&=S(z)R_1^{\frac{1}{2}}C_2^{\frac{1}{2}},
\\
L_2^{\frac{1}{2}}S(z)&=C_1^{\frac{1}{2}}S(z)L_1^{\frac{1}{2}}L_2^{\frac{1}{2}}.
\end{split}
\end{align}

\begin{theorem}\label{thm.hat.Psi.syst}
The $\Psi$-system of Theorem \ref{thm.Psi.syst} extends to a $\hat{\Psi}$-system with $L^\frac{1}{2},R^\frac{1}{2}$ and $C^\frac{1}{2}$ given in Proposition \ref{sqrtLRC}. 
\end{theorem}
\begin{proof}

We can easily show that the set of equalities \eqref{psi.syst.eq1} holds true using Propositions \ref{det.A.et.B} and \ref{sqrtLRC}.
Regarding the set of equalities \eqref{psi.syst.eq2}, by \cite[Lemma 13]{geer2012tetrahedral}, we can observe that these equatilies hold true, for any $z\in\bbRzu$, without the square roots. 
It is then easy to see that they also hold true with $L^{\frac{1}{2}},R^{\frac{1}{2}}$ and $C^{\frac{1}{2}}$, for any $z\in\bbRzu$.
\end{proof}

\subsection{Existence of 3-manifolds invariant from the \texorpdfstring{$\hat{\Psi}$}{\unichar{"03A8}\unichar{"0302}}-system}

In order to prove that \cite[Conjecture 42]{geer2012tetrahedral} holds true, we consider the operator 
\begin{align}\label{def.frakq}
\frakq=R^{\frac{1}{2}}BL^{\frac{1}{2}}BL^{-\frac{1}{2}}C^{-\frac{1}{2}}\in\End(\calH),
\end{align}
and show that there exists $q\in\bbC$ such that $\frakq = q\Id_{\check{\mathcal{H}}}\oplus\, q^{-1}\Id_{\hat{\mathcal{H}}}$. Using Propositions \ref{det.A.et.B} and \ref{sqrtLRC}, a straightforward computation leads to the following results.
\begin{lemma}\label{prop.frak.q}
For all $x\in\bbRzu$ and all $\alp\in\bbZ_N^2$, the operator $\frakq\in\End(\calH)$ defined by \eqref{def.frakq} is given by 
\begin{align*}
\frakq e_\alp(x)=\om^{\frac{1}{2}+\frac{3(\mathtt{a}-\mathtt{b})}{2}}e_{\alp}(x) \quad \text{and} \quad \frakq \bar{e}_\alp(x)=\om^{-\frac{1}{2}+\frac{3(\mathtt{b}-\mathtt{a})}{2}}\bar{e}_{\alp}(x),
\end{align*}
with $\mathtt{a},\mathtt{b}\in\bbZ_N$ given in Proposition \ref{det.A.et.B}.
\end{lemma}
%

The direct consequence of Theorem \ref{thm.hat.Psi.syst} and Lemma \ref{prop.frak.q} is that we can define a topological invariant of 3-manifolds using the $\hat{\Psi}$-system of Theorem \ref{thm.hat.Psi.syst}. This invariant is described in the following Section.

%






\section{State sum invariants}\label{hyperbolic.invariant}


\subsection{\texorpdfstring{$H$}{H}-triangulation of \texorpdfstring{$(M,L)$}{(M,L)}}
Let $M$ be a closed connected oriented 3-manifold. A {\it quasi-regular triangulation $\calT$ of $M$} is a decomposition of $M$ as a union of embedded tetrahedra (3-simplices) such that the intersection of any two tetrahedra is a union (possibly, empty) of several of their vertices (0-simplices), edges (1-simplices) and faces (2-simplices). Quasi-regular triangulations differ from the usual triangulations in that they may have tetrahedra meeting along several vertices, edges, and faces. Note that each edge of a quasi-regular triangulation has two distinct endpoints. In the sequel, we denote $\Lambda_i(\calT)$ the set of $i$-simplices of $\calT$ for $i\in\brkt{0,1,2,3}$. 

Consider a non-empty link $L\subset M$. An {\it $H$-triangulation of $(M,L)$} is a pair $(\calT,\calL)$ where $\calT$ is a quasi-regular triangulation of $M$ and $\calL\subset\Lambda_1(\calT)$ is such that each vertex of $\calT$ belongs to exactly two edges of $\calL$ and $L$ is the union of the elements of $\calL$.
\begin{proposition}[\cite{baseilhac2004quantum}, Proposition 4.20]
For any non-empty link $L$ in $M$, the pair $(M, L)$ admits an $H$-triangulation.
\end{proposition}

$H$-triangulations of $(M,L)$ can be related by elementary moves of two types, the {\it $H$-bubble moves} and the {\it $H$-Pachner} 2-3 {\it moves}. The {\it positive} $H$-bubble move on an $H$-triangulation $(\calT ,\calL)$ starts with a choice of a face $F = v_0v_1v_2\in\Lambda_2(\calT)$, where $v_0,v_1,v_2\in\Lambda_0(\calT)$, such that at least one of its edges, say $v_0v_1$, is in $\calL$. Consider two tetrahedra of $T$ meeting along $F$. We unglue these tetrahedra along $F$ and insert a 3-ball between the resulting two copies of $F$. We triangulate this 3-ball by adding a vertex $v_3$ at its center and three edges connecting $v_3$ to $v_0, v_1$, and $v_2$. The edge $v_0v_1$ is removed from $\calL$ and replaced by the edges $v_0v_3$ and $v_1v_3$. This move can be visualized as the transformation
\begin{center}
\includegraphics[width=0.95\linewidth]{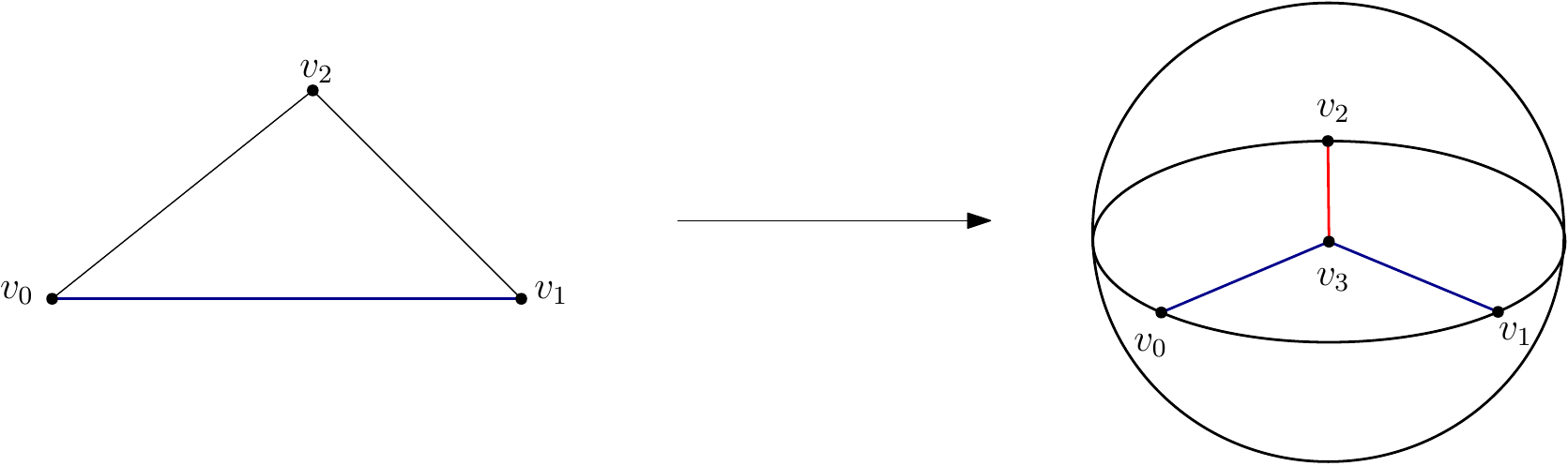}
\end{center}
where the blue edges belong to $\calL$. The inverse move is the {\it negative $H$-bubble move}. The {\it positive H-Pachner} 2-3 {\it move} can be visualized as the transformation
\begin{center}
\includegraphics[width=0.95\linewidth]{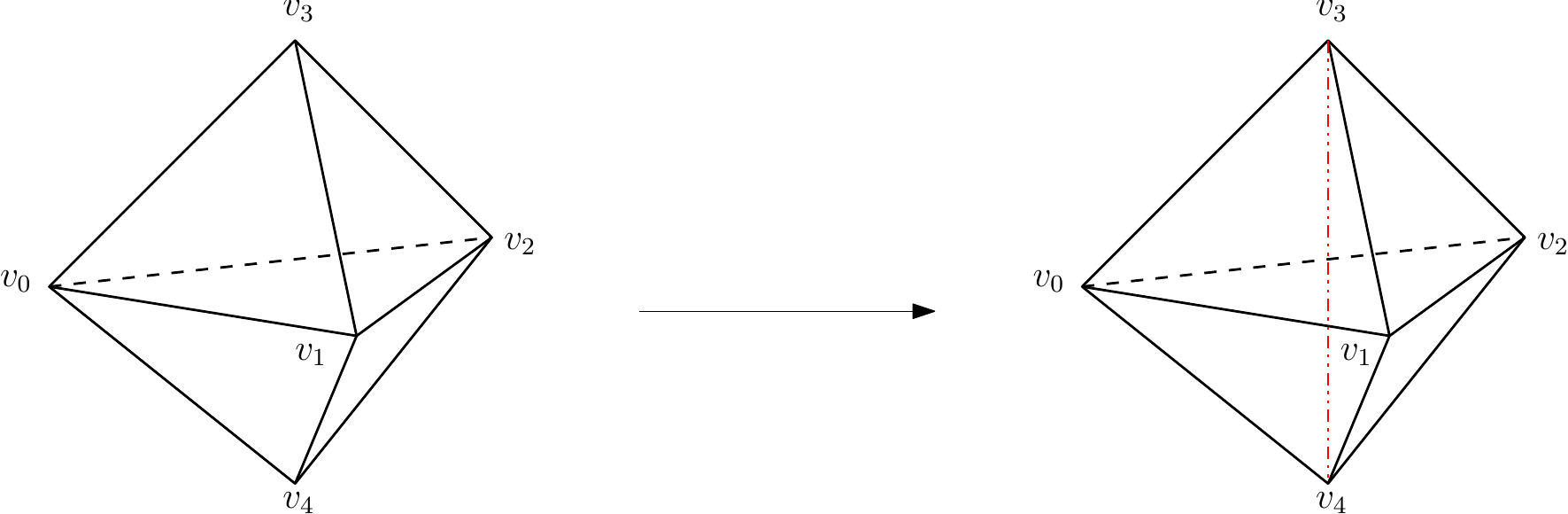}
\end{center}
This transformation preserve the set $\calL$. The inverse move is the {\it negative $H$-Pachner move}; it is allowed only when the edge common to the three tetrahedra on the right is not in $\calL$.

\begin{proposition}[\cite{baseilhac2004quantum}, Proposition 4.23]
Let $L$ be a non-empty link in $M$. Any two $H$-triangulations of $(M,L)$ can be related by a finite sequence of $H$-bubble moves and $H$-Pachner moves in the class of $H$-triangulations of (M,L).
\end{proposition}

\subsection{Charge on \texorpdfstring{$(\calT,\calL)$}{(T,L)} and \texorpdfstring{$\bbR$}{\unichar{"211D}}-coloring of \texorpdfstring{$\calT$}{T}} 

A {\it charge on $T\in\Lambda_3(\calT)$} is a map $c:\Lambda_1(T)\to\bbZ$ such that
\begin{enumerate}
	\item $c(e)=c(e')$ if $e,e'$ are opposite edges,
	\item $c(e_1)+c(e_2)+c(e_3)=1$ if $e_1,e_2,e_3$ are edges of a face of $T$.
\end{enumerate}

We denote $\Lambda_3^1(\calT)=\brkt{(T,e)\,|\, T\in\Lambda_3(\calT), e\in\Lambda_1(T)}$ and consider the obvious projection $\epsilon_\calT:\Lambda_3^1(\calT)\to \Lambda_1(\calT)$. For any edge $e$ of $\calT$, the set $\epsilon^{-1}_\calT(e)$ has $n$ elements, where $n$ is the number of tetrahedron adjacent to $e$. A {\it charge on $(\calT,\calL)$} is a map $c:\Lambda_3^1(\calT)\to\bbZ$ such that 
\begin{enumerate}

	\item the restriction of $c$ to any tetrahedron $T$ of $\calT$ is a charge on $T$,
	
	\item for each edge $e$ of $\calT$ not belonging to $\calL$ we have $\sum_{e'\in\epsilon^{-1}_\calT(e)}c(e')=2$,
	
	\item for each edge $e$ of $\calT$ belonging to $\calL$ we have $\sum_{e'\in\epsilon^{-1}_\calT(e)}c(e')=0$,
	
\end{enumerate}
Each charge $c$ on $(\calT,\calL)$ determines a cohomology class $[c]\in H^1(M,\bbZ/2\bbZ)$ (see \cite{baseilhac2004quantum} and \cite{geer2012tetrahedral} for details). Following \cite{baseilhac2004quantum}, we further require to only consider charges $c$ that are such that $[c]=0.$




A {\it $G$-coloring of $\calT$} is a map $\Phi$ from the oriented edges of $\calT$ to a group $G$ such that
 \begin{enumerate}

	\item $\Phi(-e)=\Phi(e)^{-1}$ for any oriented edge $e$ of $\calT$, where $-e$ is $e$ with the opposite orientation,
	
	\item $\Phi(e_1)\Phi(e_2)\Phi(e_3)=1$ if $e_1,e_2,e_3$ are edges of a face of $\calT$ with the following orientations:
		\begin{center}
		\includegraphics[width=0.2\linewidth]{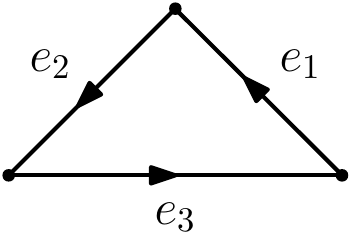}
		\end{center}
		\end{enumerate} 

\noindent

A {\it$G$-gauge} of $\calT$ is a map from the set of vertices of $\calT$ to $G$. The $G$-gauges of $\calT$ form a multiplicative group which acts on the set of $G$-colorings of $\calT$ as follows: if $\delta$ is a $G$-gauge of $\calT$ and $\Phi$ is a $G$-coloring of $\calT$, then the $G$-coloring $\delta \Phi$ is given by
\begin{align}\label{eq.gauge.rel}
(\delta \Phi)(e)=\delta\left(v_{e}^{-}\right) \Phi(e) \delta\left(v_{e}^{+}\right)^{-1}    
\end{align}
where $v_{e}^{-}$ (resp. $v_{e}^{+}$) is the initial (resp. terminal) vertex of an oriented edge $e$.

Let $\mathcal{M}(M,G)$ be the set of conjugacy classes of group homomorphisms from the fundamental group of $M$ to $G$. As explained in \cite{geer2012tetrahedral}, there is a bijective correspondence between the elements of $\mathcal{M}(M,G)$ and the $G$-colorings of $\calT$ considered up to gauge transformations. 


In what follows, we are going to consider edge orientation on $\calT$ given by a total order on $\Lambda_0(\calT)$ with the following rule: since any edge of $\calT$ has two distinct endpoints, each edge is oriented by the arrow emanating from its smallest endpoint. We also let $G$ be the additive group $(\bbR,+)$ and we consider $\bbR$-colorings on $\calT$.

\begin{proposition}[\cite{baseilhac2004quantum}, Theorem 4.7]
For any $H$-triangultaion $(\calT,\calL)$ of $(M,L)$, any choice of total order on $\Lambda_0(\calT)$ and any class $[\Phi]\in \mathcal{M}(M,\bbR)$ there exist a $\bbR$-coloring $\Phi$ of $\calT$ representing $[\Phi]$ and a charge $c$ on $(\calT,\calL)$ such that $[c]=0$.
\end{proposition}

\subsection{Charged $6j$-symbols and the symmetry relations}
The building blocks of the state sum invariant we are going to define are a modified version of the $6j$-symbols called {\it charged $6j$-symbols}. These operators are defined as follow:
%
for any $z\in\bbRzu$ and $a,c\in\frac{1}{2}\bbZ$ we define the charged 6$j$-symbols by 
\begin{align}
S(z|a,c)= \frakq_2^{-4ac}R_2^cR_1^{-a}S(z)L_2^{-a}R_2^{-c}
\end{align}
\begin{align}
S(z|a,c)^{-1}= \frakq_1^{4ac}R_2^{-c}L_2^{-a}S(z)^{-1}R_1^{-a}R_2^c
\end{align}

The charged $6j$-symbols verify particular set of relations called {\it symmetry relations}.

The symmetry relations of the charged $6j$-symbols are expressed with the  symmetric operators $\sfA=AL^{-\frac{1}{2}}$ and $\sfB=BR^{-\frac{1}{2}}$ in $\End(\calH)$. Similarly to other operators $\End(\calH)$ considered thus far, we consider, for any $\alp,\bet\in\bbZ_N^2$,  
the functions
\begin{align*}
\sfA_{\alp,\bet},\bar{\sfA}_{\alp,\bet},\sfB_{\alp,\bet},\bar{\sfB}_{\alp,\bet}  : \bbRzu\to \bbC
\end{align*}
satisfying 
\begin{align*}
\sfA h_{p,q}(u_\alp)=\sum_{\bet\in\bbZ_N^2}\sfA_{\alp,\bet}\pa{\tfrac{q}{p+q}}h_{p,q}(u_\bet), 
\\
\sfA\bar{h}_{p,q}(\bar{u}_\alp)=\sum_{\bet\in\bbZ_N^2}\bar{\sfA}_{\alp,\bet}\pa{\tfrac{q}{p+q}}\bar{h}_{p,q}(\bar{u}_\bet), 
\\
\sfB h_{p,q}(u_\alp)=\sum_{\bet\in\bbZ_N^2}\sfB_{\alp,\bet}\pa{\tfrac{q}{p+q}}h_{p,q}(u_\bet), 
\\
\sfB\bar{h}_{p,q}(\bar{u}_\alp)=\sum_{\bet\in\bbZ_N^2}\bar{\sfB}_{\alp,\bet}\pa{\tfrac{q}{p+q}}\bar{h}_{p,q}(\bar{u}_\bet),
\end{align*}
for all admissible pairs $(p,q)\in(\bbRz)^2$. A straightforward computation, using Lemmas \ref{det.A.et.B} and \ref{sqrtLRC}, leads us to the followings results.

\begin{lemma}
The operators $\sfA,\sfB\in\End(\calH)$ are symmetric involutions. Moreover, for all $\alp,\bet\in\bbZ_N^2$ and all $x\in\bbRzu$, we have 
\begin{align*}
\sfA_{\alp,\bet}(x)  =\sfA_{\alp,\bet} & =  \epsilon_N^2\frac{\del_{\alp_2,-\bet_2}}{N}\om^{-\frac{1}{2}\pa{t(\alp_1-\alp_2-\bet_1)^2+\alp_1(\alp_1+t-1)+\bet_1(\bet_1+t-1)-1}+\mathtt{a}},\\
\bar{\sfA}_{\alp,\bet}(x)  =\bar{\sfA}_{\alp,\bet} & =  \epsilon_N^{-2}\del_{\alp_2,-\bet_2}\om^{\frac{1}{2}\pa{t(\alp_1-\alp_2-\bet_1)^2+\alp_1(\alp_1+t-1)+\bet_1(\bet_1+t-1)-1}-\mathtt{a}},\\
\sfB_{\alp,\bet}(x) =\sfB_{\alp,\bet} & = \del_{\alp,-\bet}\om^{-\frac{1}{2}\pa{(t+1)(\alp_1-\alp_2)^2+\alp_1^2}+\mathtt{b}},\\
\bar{\sfB}_{\alp,\bet}(x)  =\bar{\sfB}_{\alp,\bet} & =  \del_{\alp,-\bet}\om^{\frac{1}{2}\pa{(t+1)(\alp_1-\alp_2)^2+\alp_1^2}-\mathtt{b}},
\end{align*}
with $\mathtt{a},\mathtt{b}\in\bbZ_N$ given in Proposition \ref{det.A.et.B}. 
\end{lemma}

Using the previous Lemma, Remark \ref{rmk.S.as.6j} and Proposition \ref{prop.frak.q}, the following Proposition is a direct adaptation of the formulas (50), (51) and (52) in \cite{geer2012tetrahedral}. Note that we use $\equiv$ to denote equality up to multiplication by an integer power of $\om$.
\begin{proposition}[The symmetry relations]
The charged $6j$-symbols verify the following symmetry relations
\begin{align*}
\Big \langle \bar{u}_\alp\ot \bar{u}_\nu \Big| S(z|a,c) \Big| u_\mu\ot u_\bet \Big \rangle
&\equiv
\sum_{\alp',\mu'\in\bbZ_N^2} \sfA_{\alp,\alp'}\bar{\sfA}_{\mu,\mu'}\Big \langle \bar{u}_{\mu'}\ot \bar{u}_\nu \Big| S\pa{\tfrac{z}{z-1}|a,b}^{-1} \Big| u_{\alp'}\ot u_\bet \Big \rangle \\
\Big \langle \bar{u}_\alp\ot \bar{u}_\mu \Big| S(z|a,c) \Big| u_\bet\ot u_\nu \Big \rangle 
&\equiv
\sum_{\alp',\nu'\in\bbZ_N^2} \bar{\sfA}_{\nu,\nu'}\sfB_{\alp,\alp'}\Big \langle \bar{u}_{\mu}\ot \bar{u}_{\nu'} \Big| S\pa{z^{-1}|b,c}^{-1} \Big| u_{\alp'}\ot u_\bet \Big \rangle\\
\Big \langle \bar{u}_\mu\ot \bar{u}_\bet \Big| S(z|a,c) \Big| u_\alp\ot u_\nu \Big \rangle
&\equiv
\sum_{\bet',\nu'\in\bbZ_N^2} \sfB_{\bet,\bet'}\bar{\sfB}_{\nu,\nu'}\Big \langle \bar{u}_{\mu}\ot \bar{u}_{\nu'} \Big| S\pa{\tfrac{z}{z-1}|a,b}^{-1} \Big| u_{\alp}\ot u_{\bet'} \Big \rangle
\end{align*} 
for any $a,b,c\in\frac{1}{2}\bbZ$ such that $a+b+c=\frac{1}{2}$.
\end{proposition}
%
A topological interpretation of these relations is given in \cite[Appendix]{geer2012tetrahedral}.

\subsection{State sum invariant}


\noindent


Let $(\calT,\calL)$ be an $H$-triangulation of $(M,L)$. Fix a total order on $\Lambda_0(\calT)$ and consider a $\bbR$-coloring $\Phi$ of $\calT$, a charge $c$ on $(\calT,\calL)$ that is such that $[c]=0$ and a map $\alp: \Lambda_2(\calT)\to\bbZ_N^2$.
From this data, we define the state sum as follows.
Let $T$ be a tetrahedron of $\calT$ with ordered vertices $v_0,v_1,v_2,v_3$. We say that $T$ is {\it right oriented} if the vertices $v_0,v_1$ and $v_2$ go round in the counter-clockwise direction when we look at them from $v_3$ in the increasing order. Otherwise, $T$ is {\it left oriented}. Set 
\begin{align}\label{def.param.6j}
    p=\Phi(\overrightarrow{v_0v_1}),\ \ q=\Phi(\overrightarrow{v_1v_2}), \ \ r=\Phi(\overrightarrow{v_2v_3})
\end{align}
where $\overrightarrow{v_iv_j}$ is the oriented edge of $T$ going from $v_i$ to $v_j$. 
Then set  
\begin{align}\label{ideal.tetra.param}
    z=\frac{pr}{(p+q)(q+r)}\in\bbRzu
\end{align}
%
and denote $c_{ij}=\frac{1}{2} c(v_iv_j)$ and $\alp_i=\alp(v_jv_kv_l)$ for $\brkt{i,j,k,l}=\brkt{0,1,2,3}$. We associate the following matrix element to the tetrahedron $T$
\begin{align}  
T(\Phi,c,\alp)=\begin{footnotesize}\left\{\begin{array}{ll}
						\Big \langle \bar{u}_{\alp_2}\ot \bar{u}_{\alp_0} \Big| S(z|c_{12},c_{23}) \Big| u_{\alp_3}\ot u_{\alp_1} \Big \rangle & \text{if } T\text{ is right oriented},  \\
						\Big \langle \bar{u}_{\alp_3}\ot \bar{u}_{\alp_1} \Big| S(z|c_{12},c_{23})^{-1} \Big| u_{\alp_2}\ot u_{\alp_0} \Big \rangle & \text{if } T\text{ is left oriented}. \\
				\end{array}
			\right.\end{footnotesize}
\end{align}
and we define the state sum as follows
\begin{align}\label{state.sum}
\sfK_N(\calT,\calL,\Phi,c)=N^{-2|\Lambda_0(\calT)|}\sum_{\alp}\prod_{T\in\calT}T(\Phi,c,\alp)\in\bbC.
\end{align}
\begin{theorem}\label{main.thm.inv}
Up to multiplication by integer powers of $\om$, 
$\sfK_N(\calT,\calL,\Phi,c)$ depends only on the isotopy class of $L$ in $M$ and the (conjugacy) class $[\Phi]\in \mathcal{M}(M,\bbR)$ (and does not depend on the choice of $\Phi$ and $c$ in their respective classes, the $H$-triangulation $(\calT,\calL)$ of $(M,L)$, and the ordering of the vertices of $\calT$).
\end{theorem}
\begin{proof}
Using Lemma \ref{prop.frak.q}, the statement of the Theorem is a direct adaptation of Theorem 29 of \cite{geer2012tetrahedral}.
\end{proof} 


\subsection{\texorpdfstring{$\mathfrak{sl}_3$}{sl_3} Kashaev's invariant}


The value $z$ in \eqref{ideal.tetra.param} entering the (charged) $6j$-symbols has a natural geometric interpretation. Indeed, following \cite{baseilhac2002qhi}, given a $\bbR$-coloring $\Phi$ of $\calT$ and a tetrahedron $T$ of $\calT$ with ordered vertices $v_0,v_1,v_2,v_3$, we consider
\begin{align*}
    p_0 = \Phi(\overrightarrow{v_0v_1}) \cdot \Phi(\overrightarrow{v_2v_3}), \quad p_1 = \Phi(\overrightarrow{v_1v_2}) \cdot \Phi(\overrightarrow{v_0v_3}), \quad p_2 = -\Phi(\overrightarrow{v_0v_2}) \cdot \Phi(\overrightarrow{v_1v_3}), 
\end{align*}
and 
\begin{align*}
& w_0=-p_2/p_0,\quad w_1=-p_0/p_1, \quad w_2=-p_1/p_2.
\end{align*}
We can observe that the following relations hold true
\begin{align*}
    w_1=\dfrac{1}{1-w_0}, \ \quad \ w_2=\dfrac{w_0-1}{w_0}.
\end{align*}
These relations imply that $w_0,w_1$ and $w_2$ are shape parameters of (flat) ideal tetrahedra in $\mathbb{H}^3$ which means that these numbers and their inverses determine a congruence class of ideal tetrahedra in $\mathbb{H}^3$. Our claim on $z$ in \eqref{ideal.tetra.param} follows from the fact that $z=w_0^{-1}$. 


In the particular case where $[\Phi]\in\mathcal{M}(M,\bbR)$ is the class of the trivial representation, one can even see that $w_0,w_1$ and $w_2$ are actually cross-ratios. Indeed, it follows directly from the fact that \eqref{def.param.6j} is given, in that case, by
\begin{align*}
    \Phi(\overrightarrow{v_0v_1}) = \delta(v_0)-\delta(v_1),\ \ \Phi(\overrightarrow{v_1v_2})= \delta(v_1)-\delta(v_2), \ \ \Phi(\overrightarrow{v_2v_3})= \delta(v_2)-\delta(v_3)
\end{align*}
where $\delta$ is any $\bbR$-gauge  of $\calT$. Let us denote by $\delta_{i}$ the value $\delta(v_i)$, 
 for $i\in\{0,1,2,3\}$. Then $z$ in  \eqref{ideal.tetra.param} is the following cross-ratio
\begin{align*}
    z = \dfrac{(\delta_{0}-\delta_{1})(\delta_{2}-\delta_{3})}{(\delta_{0}-\delta_{2})(\delta_{1}-\delta_{3})} = [\delta_0:\delta_3 :\delta_1 :\delta_2]. 
\end{align*}
In this special case, the state sum \eqref{state.sum} defines a link invariant that correspond to a $\mathfrak{sl}_3$ version of Kashaev's invariants defined in \cite{kashaev1994quantum} and \cite{kashaev1995link}. The construction of a $\mathfrak{sl}_3$ version of Baseilhac and Benedetti's quantum hyperbolic invariants with $PSL(2,\bbC)$-Characters is left for further research.




\section*{Acknowledgment}
The author wishes to thank Prof. R. Kashaev for valuable discussion and support as well as Prof. H. Schenck for useful comments. 

\bibliographystyle{amsplain}
\bibliography{mybibfile}

\end{document}